\documentclass[12pt,
reqno]{amsart}

\usepackage{ccfonts}
\usepackage[T1]{fontenc}
\usepackage[cp1251]{inputenc}
\usepackage{amsmath,amsxtra,amssymb,eufrak}
\usepackage[all]{xy}
\usepackage{mathrsfs}
\usepackage{paralist}
\usepackage[active]{srcltx} 

\newtheorem{thm}{Theorem}[section]
\newtheorem{lm}[thm]{Lemma}

\newtheorem{pr}[thm]{Proposition}

\theoremstyle{definition}

\newtheorem{df}[thm]{Definition}
\newtheorem{exm}[thm]{Example}
\newtheorem{rem}[thm]{Remark}
\newtheorem{que}{Question}
\numberwithin{equation}{section}

\DeclareMathOperator{\Ker}{Ker}
\DeclareMathOperator{\LinC}{Lin_{\mathbb C}}

\newenvironment{mycompactenum}{\pltopsep=5pt\begin{compactenum}[\upshape (i)]}%
{\end{compactenum}}

\renewcommand{\le}{\leqslant}
\renewcommand{\ge}{\geqslant}

\let \al         =\alpha
\let \be         =\beta
\let \ga         =\gamma
\let \de         =\delta
\let \ze         =\zeta
\let \io         =\iota

\let \la         =\lambda
\let \si         =\sigma
\let \up         =\upsilon
\let \om         =\omega

\let \Ga         =\Gamma
\let \De         =\Delta

\let \Si         =\Sigma

\title{Holomorphically finitely generated Hopf algebras and quantum Lie groups}
\author{O. Yu. Aristov}
\email{aristovoyu@inbox.ru}
\address {Institute for Advanced Study in Mathematics, Harbin Institute of Technology,  Harbin 150001, China}
\keywords{Holomorphically finitely generated algebra, Topological Hopf algebra,  Cartier's theorem, Quantum group, Drinfeld--Jimbo algebra, Arens--Michael envelope, Complex Lie group, Holomorphic function of exponential type.}
\subjclass[2000]{Primary 17B37, 46H35, 22E10,  Secondary  32A38, 16S38, 58B34}

\begin{document}
 \maketitle
 \markright{Holomorphically finitely generated Hopf algebras}
\begin{abstract}
We study topological Hopf algebras that are holomorphically finitely generated (HFG) as Fr\'echet Arens--Micheal algebras in the sense of Pirkovskii. Some of them, but not all, can be obtained from affine Hopf algebras by applying the analytization functor. We show that a commutative HFG  Hopf algebra is always  an algebra of holomorphic functions on a complex Lie group (actually a Stein group),  and prove that the corresponding categories are equivalent.
With a compactly generated complex Lie group~$G$, Akbarov associated a cocommutative topological Hopf algebra, the algebra ${\mathscr A}_{exp}(G)$ of exponential analytic functionals. We show that it is HFG but not every cocommutative  HFG Hopf algebra is of this form. In the case when $G$ is connected, using previous results of the author we establish a theorem on the analytic structure of ${\mathscr A}_{exp}(G)$. It depends on the large-scale geometry of $G$. We also consider some interesting examples including complex-analytic analogues of classical $\hbar$-adic quantum groups.
\end{abstract}

{\small\it
\hfill To the memory of Aleksander Jurchenko,}

{\small\it
\hfill who explained me that Mathematics is}

{\small\it
\hfill  the most beautiful thing in the Universe}

\section*{Introduction}

We begin to study topological Hopf algebras satisfying a certain finiteness condition.  Whereas Hopf algebras are a classical object of research, their topological versions outside the $C^*$-algebraic context  have so far received less attention than they deserve. Here we impose additional limitations  in the definition that enable us to prove effective results and consider some encouraging examples satisfying these restrictions, in particular, those of quantum group nature. It seems to the author that this approach may even open up a new field of research at the intersection of functional analysis, noncommutative geometry and quantum group theory.

Here are the main questions that are under consideration.

1.~What is a correct analogue of the notion of an affine Hopf algebra in the context of complex-analytic noncommutative geometry?

2.~What is a quantum complex Lie group?

(Recall that `affine' here  stands for `finitely generated as an algebra.' Speaking on a quantum group we mean an object in the category dual to some category of Hopf algebras and assume that this object in some indistinct sense is a deformation of a commutative or cocommutative Hopf algebra.)

An answer to the first question, which is proposed here, is: a topological Hopf algebra with the finiteness condition introduced for a topological associative algebra by Pirkovskii in \cite{Pi14,Pi15}. (Such an algebra is called holomorphically finitely generated or simply HFG.) The main point of this paper is that the class of HFG Hopf algebras is sufficiently large to contain many non-trivial examples and small enough to prove some general, in particular, structure theorems.  We establish such a structure result in the commutative case and hope that, under additional restrictions, this can be done in the cocommutative case.

As regards the second question, we consider complex-analytic forms of some Hopf algebras originating from the quantum group theory. For example, standard deformations of universal enveloping algebras are usually considered in the
$\hbar$-adic form because the definitions include such function as the hyperbolic sine, which is not a polynomial. But this approach is somewhat extreme since the hyperbolic sine is just an entire function and it is not necessary to consider it as a formal power series. On the other hand, the relations  containing entire functions are inherent in our approach and, in author's opinion, complex-analytic forms of quantum algebras are not only natural but easier to manage than $\hbar$-adic. We discuss some examples that are not originated from affine Hopf algebras in Section~\ref{eqclg}.

Our purposes are:

-~To show that the category of commutative  HFG Hopf algebras is anti-equivalent to  the category of Stein groups; see Section~\ref{cHHFGa}.

-~To prove some general results on cocommutative  HFG Hopf algebras; see Section~\ref{cocHHFGamt}.

-~To consider examples of cocommutative HFG Hopf algebras and complex-analytic forms of some classical (and also less famous)  quantum groups; see Sections~\ref{Cocoexm}  and~\ref{eqclg}, respectively.

Note that  in this article we do not discuss the real case. It seems that a gap between noncommutative differential and noncommutative complex-analytic geometries is wider and deeper than between their commutative  counterparts. To find a similar definition of quantum group in the framework of noncommutative differential geometry  one needs other ideas and methods.

A detailed account on definitions and results follows.

\subsection*{Topological Hopf  algebras}
By a  topological  Hopf algebra we mean an object of some monoidal category of locally convex spaces. Usually spaces are assumed complete and the complete projective tensor product bifunctor $(-)\mathbin{\widehat{\otimes}} (-)$ is taken as a monoidal product.
Topological Hopf algebras in this sense, which are called \emph{Hopf $\mathbin{\widehat{\otimes}}$-algebras},  were studied in last decades; see, e.g., \cite{Lit2,BFGP,Pir_stbflat}. But the whole class of Hopf $\mathbin{\widehat{\otimes}}$-algebras is rather wide to prove structure theorems effectively.

We use the finiteness condition introduced by Pir\-kov\-skii.
Recall that for every complex space
$(X,\mathcal{O}_X)$ the set $\mathcal{O}(X)$ of global sections can be considered as a Fr\'echet Arens--Michael algebra
and that a  Fr\'echet algebra is called a \emph{Stein algebra} if it
is topologically isomorphic to $\mathcal{O}(X)$  for some Stein space $(X,\mathcal{O}_X)$.
In~\cite{Fo67} Forster proved that the functor $(X, \mathcal{O}_X)\mapsto \mathcal{O}(X)$ is an
anti-equivalence between the category of Stein spaces and the
category of Stein algebras. Pirkovskii defined an HFG (stands for 'holomorphically finitely generated') algebra as an algebra having finitely many `holomorphic generators' (see the exact meaning in Section~\ref{defssec})
and improved  Forster's result by showing that
a commutative Fr\'{e}chet Arens--Michael algebra is HFG if
and only if it is topologically isomorphic to $\mathcal{O}(X)$ for some
Stein space $(X, \mathcal{O}_X)$ of finite embedding dimension.
As a corollary,  the functor $(X, \mathcal{O}_X)\mapsto \mathcal{O}(X)$ is an anti-equivalence
between the category of Stein spaces of finite embedding dimension
and the category of commutative HFG algebras \cite[Theorems~3.22,
3.23]{Pi15} (announced in  \cite[Theorem~2.9]{Pi14}). Following the standard pattern, one can consider objects of the category  dual to that of HFG algebras  as `noncommutative Stein spaces'.

Here we apply Pirkovskii's idea to Hopf $\mathbin{\widehat{\otimes}}$-algebras by considering those of them which are HFG algebras.

\subsection*{The commutative case}

In view of famous Cartier's theorem, which asserts that any affine commutative Hopf algebra is the Hopf algebra of regular functions on some algebraic group, it is natural to guess that a similar result holds for an HFG  Hopf algebra. Indeed, this is the case: any commutative  HFG Hopf algebra is $\mathcal{O}(G)$, the set of holomorphic functions
on some complex Lie group~$G$ with standard operations. In fact, $G$ can be assumed to be a Stein group. Moreover, there is an  anti-equivalence of the corresponding categories; see Theorem~\ref{antiSTcoH} below. The proof is more or less standard and follows  the argument for Cartier's theorem. The auxiliary results for Theorem~\ref{antiSTcoH} are collected in Appendix~\ref{ap:top}.

\subsection*{The cocommutative case}

In \cite{Ak08} Akbarov  considered a cocommutative   topological Hopf algebra associated with a complex Lie group~$G$. The construction is in two steps. First, we take ${\mathscr A}(G)$, the Hopf $\mathbin{\widehat{\otimes}}$-algebra of \emph{analytic functionals} on~$G$, which is defined as the strong dual space of $\mathcal{O}(G)$; for details see \cite{Pir_stbflat} and the references therein. The second step is to apply the Arens--Michael envelope functor to~${\mathscr A}(G)$. (See the definition of the Arens--Michael envelope in Section~\ref{defssec}; it can be considered as an analytization functor and a natural bridge between noncommutative affine algebraic geometry and noncommutative complex-analytic geometry \cite{Pir_qfree}.)  The resulting cocommutative Arens--Michael Hopf $\mathbin{\widehat{\otimes}}$-algebra $\widehat{\mathscr A}(G)$ can be considered in some sense as dual to the commutative Arens--Michael Hopf $\mathbin{\widehat{\otimes}}$-algebra $\mathcal{O}(G)$; see details in [ibid.] and the paper~\cite{Ar20+B} of the author. Since $\mathcal{O}(G)$ is HFG, the natural question arises: whether this property inherits under duality?

The main result of Section~\ref{cocHHFGamt},
Theorem~\ref{cicHFGgen}, provides the following answer: the Arens--Michael envelope $\widehat{\mathscr A}(G)$ of ${\mathscr A}(G)$ is an HFG Hopf  algebra when~$G$ is compactly generated and, moreover,  $\widehat{\mathscr A}(G)$ is a~Fr\'echet algebra or even an~HFG algebra only when this condition holds. Note that if~$G$ is simply connected, the assertion easily follows from the fact that
$\widehat{\mathscr A}(G)$ coincides with the Arens--Michael envelope of $U(\mathfrak{g})$, the universal enveloping algebra of the Lie algebra of~$G$. In the general case, the proof needs some specific tools of the theory of locally convex spaces, e.g., Pt\'{a}k's open mapping theorem. Furthermore, the argument for Theorem~\ref{cicHFGgen} uses the technique of analytic free products,
necessary information on which is contained in Appendix.

We also show that the class of cocommutative HFG Hopf algebras is not exhausted
by algebras of the form $\widehat{\mathscr A}(G)$, where~$G$ runs over compactly generated Lie groups. In particular, $\widehat{U}(\mathfrak{f}_n)$, where $\mathfrak{f}_n$ is the free Lie algebra in~$n$ generators and $n>1$, is not of this form;
see Theorem~\ref{fgenLiea} and Example~\ref{freeLa}. So if one wants to find a characterization similar to the commutative case, some additional assumptions are required.

\subsection*{Examples}
The second half of the paper is devoted primarily to examples. In Section~\ref{Cocoexm} we consider that of cocommutative  HFG Hopf algebras associated with  Lie groups (including the discrete case) and Lie algebras. The main conclusion that we obtain from considering these examples is that the analytic structure of a  cocommutative  HFG Hopf algebra can be rather complicated. For example, the explicit description of $\widehat{\mathscr A}(G)$ for connected~$G$  that is given in Theorem~\ref{nilpenv} (in the nilpotent case) and Theorem~\ref{genpenv} (in the general case)
reflects the large-scale geometry of $G$ in an appropriate way. Preliminary work  has been done in  articles \cite{ArAMN,Ar19} of the author.

In Section~\ref{eqclg} some examples of quantum complex Lie groups are collected. We consider deformations of the Arens--Micheal envelope of the universal enveloping algebra of $\mathfrak{s}\mathfrak{l}_2$  and the Lie algebra of the `az+b' group as well as a deformation of the algebra of  holomorphic functions on the simply-connected cover of the `az+b' group. These examples are chosen because there are no naturally defined affine Hopf algebras for them. And the author stresses again that their complex-analytic forms are more natural than $\hbar$-adic.

\subsection*{On $C^*$-algebraic and bornological quantum groups}
The approach developed here is closer to classical quantum group theory (described, e.g., in \cite{Ka95,KSc}) than to $C^*$-algebraic quantum group theory (main points  of the latter theory  can be found in~\cite{Ti08}). A more recent study of bornological quantum groups was started in \cite{Vo08} (see also \cite{RY21}). The essential features of both $C^*$-algebraic and bornological approaches are the use of involutive multiplier Hopf algebras instead of standard ones and  the important role that played by invariant weights (whose existence can be postulated or proved).

We do not need multipliers here because all algebras under consideration are unital. On the other hand, even in the commutative case, the inclusion of an involution and invariant weights in the Hopf $\mathbin{\widehat{\otimes}}$-algebra context are problematic:
the complex conjugate of a non-constant holomorphic function is not holomorphic and  a non-trivial entire function is not integrable with respect to a Haar measure. Nevertheless, the most ambitious goal of future research is to associate a $C^*$-algebraic quantum group with every complex-analytic quantum group (as well as a quantum group in the usual sense). But at the moment the author does not see how it can be done and the relationship between the old theories and the new approach developed here is unclear.

Note also that $C^*$-algebraic theory uses invariant weights as a basis for duality. On the contrary, Hopf $\mathbin{\widehat{\otimes}}$-algebra theory does not need them to construct duals. This was first shown by Akbarov in~\cite{Ak08}. So life without invariant weights is possible. For some results on duality in this context see the author's papers \cite{Ar20+B,Ar22B}.

\subsection*{Open questions}
Here are a short list of questions.

First, note  that each  HFG Hopf algebra considered in this paper has invertible antipode. On the other hand, it is not hard to find an Arens--Michael Hopf $\mathbin{\widehat{\otimes}}$-algebra having non-invertible antipode. For example, one can consider famous Takeuchi's example $T=H(M_n^*)$ ($n\ge 2$) of a Hopf algebra, whose  antipode is  not invertible  \cite[Theorem~11]{Tak71}. Moreover, it can be proved that the antipode of the Arens--Michael envelope~$\widehat T$ is also non-invertible.

Note that the Arens--Michael envelope of an affine Hopf algebra is HFG; see Proposition~\ref{ANenvaff} below.
But $T$ is not affine and, moreover, it can be shown that $\widehat T$ is not HFG.
So far as known to the author, the question whether an affine Hopf algebra has invertible antipode is still open.
For some results and conjectures on the invertibility we refer the reader to~\cite{Sk06,LOWY18}. Summing the discussion, we conclude that the following question is natural.

\begin{que}
Is the antipode of any HFG Hopf  algebra invertible?
\end{que}

Now we take a look on the functor $H\mapsto \widehat H$. Of course, the Arens--Michael enveloping homomorphism for an associative algebra may have non-trivial kernel.
This is also possible for Hopf $\mathbin{\widehat{\otimes}}$-algebras, for example, this is the case when $H={\mathscr A}(G)$, where~$G$ is a non-linear Lie group~\cite{Ar20+B}. But the following question is open.

\begin{que}
Is the Arens--Michael enveloping homomorphism injective for every affine Hopf algebra? (Recall that in this case $\widehat H$  is HFG.)
\end{que}

\begin{que}
Suppose that $H_1$ and $H_2$ are affine Hopf algebras such that there is an HFG Hopf algebra isomorphism $\widehat H_1\cong \widehat H_2$. Is there a Hopf algebra isomorphism $H_1\cong H_2$?
\end{que}

The author is grateful to Sergei Akbarov, Andrei Domrin and Aleksei Pirkovskii  for valuable and stimulating discussions. The author  would also like to thank the referees for suggestions that improved the presentation and also for pointing out inaccuracies in the proofs of Theorem~\ref{antiSTcoH}, Propositions~\ref{Frfgen}, \ref{primfordis} and~\ref{af1qco},  as well as in Example~\ref{GLsm}.

 \section{Definitions}
\label{defssec}

We consider vector spaces, algebras and Hopf algebras over $\mathbb{C}$.
Most of algebras in this article are unital; exceptions are semigroup algebras in Section~\ref{cocHHFGamt}.

Recall that an \emph{Arens--Michael algebra} is a complete topological algebra such that its topology
can be determined by a system of submultiplicative seminorms $\|\cdot\|$, i.e., satisfying $\|ab\|\le\|a\|\,\|b\|$ for every~$a$ and~$b$ \cite{X2}. (An alternative terminology is `complete m-convex algebra' \cite{Fra05}.)

An \emph{Arens--Michael envelope} of a topological algebra~$A$
 is  a pair $(\widehat A, \io_A)$, where
$\widehat A$ is an Arens--Michael algebra  and $\io_A$ is a
continuous homomorphism $A \to \widehat A$, such that for any
Arens--Michael algebra $B$ and for each continuous homomorphism
$\varphi\!: A \to B$ there exists a unique continuous homomorphism
$\widehat\varphi\!:\widehat A \to B$ making the diagram
\begin{equation*}
  \xymatrix{
A \ar[r]^{\io_A}\ar[rd]_{\varphi}&\widehat A\ar@{-->}[d]^{\widehat\varphi}\\
 &B\\
 }
\end{equation*}
commutative
\cite[Chapter~5]{X2}.
It is not hard to see that~$\widehat A$  is the completion of~$A$ with respect to the topology determined by
all continuous submultiplicative seminorms.

By definition, Joseph Taylor's \emph{algebra ${\mathscr F}_n$ of free entire functions}, where $n\in\mathbb{N}$, is
the Arens--Michael envelope of the free algebra over~$\mathbb{C}$  in~$n$ generators.
The algebra ${\mathscr F}_n$  admits the following explicit description:
\begin{equation*}
{\mathscr F}_n= \Bigl\{
a=\sum_{\alpha\in W_n} c_\alpha \ze_\al\!:\,  \| a\|_\rho\!:=\sum_{\alpha} |c_\alpha|\, \rho^{|\alpha|}<\infty
\;\,\forall \rho>0\};
\end{equation*}
here  $\ze_1,\ldots,\ze_n$ are the generators, $W_n$ is the set of words in $1,\ldots,n$ and $\ze_\al$ denotes the corresponding product of generators. See details
in~\cite{T2,Ta73}.

A Fr\'echet Arens--Michael algebra $A$ is called \emph{holomorphically
finitely generated} (HFG for short) if it is holomorphically generated by a finite subset~$S$ \cite[Definition~3.16]{Pi15}.
This means that each $a\in A$ can be represented as a free entire function in elements of~$S$. The equivalent definition says that $A$ is HFG if it is topologically isomorphic to
${\mathscr F}_n/I$ for some $n\in \mathbb{N}$ and for some closed two-sided
ideal $I$ of ${\mathscr F}_n$ \cite[Proposition~3.20]{Pi15}.  Note that each HFG algebra is a nuclear space
\cite[Corollary~ 3.21]{Pi15}.

We consider the category of complete locally convex spaces. Endowed with the complete projective tensor product bifunctor $(-)\mathbin{\widehat{\otimes}} (-)$ it becomes  a~braided monoidal category. A Hopf algebra in this category  is called a~\emph{Hopf $\mathbin{\widehat{\otimes}}$-algebra}.  In detail, a~Hopf $\mathbin{\widehat{\otimes}}$-algebra is a $\mathbin{\widehat{\otimes}}$-bialgebra $(H,m,u,\De,\varepsilon)$ endowed with a continuous linear map~$S$ that satisfies the antipode axiom, i.e.,  the following diagram commutes:
\begin{equation*}
  \xymatrix{
 H\mathbin{\widehat{\otimes}} H\ar[d]^{S\mathbin{\widehat{\otimes}} 1}&\ar[l]_{\De}H\ar[d]^{u\varepsilon}\ar[r]^{\De}
 &H\mathbin{\widehat{\otimes}} H\ar[d]^{1\mathbin{\widehat{\otimes}} S}\\
H\mathbin{\widehat{\otimes}} H\ar[r]_m&H&H\mathbin{\widehat{\otimes}} H\ar[l]^m \,;}
\end{equation*}
here $H$ is a $\mathbin{\widehat{\otimes}}$-algebra with respect to the multiplication $m$ and unit $u\!:\mathbb{C}\to H$, and a
$\mathbin{\widehat{\otimes}}$-coalgebra with respect to the comultiplication $\De$ and counit $\varepsilon\!:H\to \mathbb{C}$.
See details, e.g., in \cite[Section~2]{Pir_stbflat}.

Now we can formulate our main definition.
\begin{df}
A Hopf $\mathbin{\widehat{\otimes}}$-algebra that is an HFG algebra is called a \emph{holomorphically
finitely generated Hopf $\mathbin{\widehat{\otimes}}$-algebra} or simply an \emph{HFG Hopf algebra}.
\end{df}
(Note that this term as well as `topological Hopf algebra' cannot be divided: in general, a HFG Hopf algebra is not a usual Hopf algebra.)

Recall that a Hopf algebra is said to be \emph{affine} if it is finitely generated as an algebra. Any affine Hopf algebra is a Hopf $\mathbin{\widehat{\otimes}}$-algebra when endowed with the strongest locally convex topology.

\begin{pr}\label{ANenvaff}
If $H$ is an affine Hopf algebra, then  $\widehat{H}$ is an HFG  Hopf algebra and the corresponding map $H\to \widehat H$ is a Hopf $\mathbin{\widehat{\otimes}}$-algebra homomorphism.
\end{pr}
\begin{proof}
Since $H$ is a Hopf $\mathbin{\widehat{\otimes}}$-algebra,  so is the Arens--Michael envelope $\widehat H$; furthermore, the corresponding map $H\to \widehat H$ is a Hopf $\mathbin{\widehat{\otimes}}$-algebra homomorphism \cite[Proposition~6.7]{Pir_stbflat}.
On the other hand, $\widehat{H}$ is an HFG algebra  because $H$ is a finitely-generated algebra \cite[Proposition~7.1]{Pi15}.
\end{proof}

So we have a plenty of  HFG  Hopf algebra examples.
For instance, any finite dimensional Hopf algebra is HFG (with respect to any norm). If $\mathfrak{g}$ is a finitely-generated Lie algebra, then $\widehat U(\mathfrak{g})$ is an HFG  Hopf algebra (here $U(\mathfrak{g})$ denotes the universal enveloping algebra of~$\mathfrak{g}$).
Another example is the `function algebra' on the quantum `az+b' group studied in \cite[Sections~8.3, 8.4]{Ak08}.
In Section~\ref{eqclg} we consider examples of  HFG Hopf algebras that are  not so directly constructed.

 \section{Commutative  HFG Hopf algebras}
\label{cHHFGa}

Let $G$ be a complex Lie group. The Fr\'{e}chet algebra $\mathcal{O}(G)$ of
holomorphic functions on $G$ has the structure of a Hopf
$\mathbin{\widehat{\otimes}}$-algebra given by
$$
 \De (f)(g, h) = f(gh),\quad \varepsilon(f) = f(e),\quad  (Sf)(g) = f(g^{-1}),
$$
where $f\in\mathcal{O}(G)$ and $ g,h\in G$. Any holomorphic homomorphism $\pi\!:G_1\to G_2$, where $G_1$ and $G_2$ are complex Lie groups, induces a continuous homomorphism $F(\pi)\!:\mathcal{O}(G_2)\to \mathcal{O}(G_1)$ of Fr\'{e}chet algebras. Moreover,  the following assertion holds.
\begin{pr}
The correspondence $\pi\mapsto F(\pi)$ determines a contravariant functor $F$ from the
category of complex Lie groups to the category of commutative
Hopf $\mathbin{\widehat{\otimes}}$-algebras.
\end{pr}
The proof is straightforward.

Recall that a complex Lie group  is called a \emph{Stein group} if the
underlying complex manifold is a Stein manifold.

 \begin{thm}\label{antiSTcoH}
The restriction of $F$ to the full subcategory of Stein groups is an
anti-equivalence with the category of commutative  HFG Hopf algebras.
 \end{thm}
The argument  is similar to that for Cartier's theorem, which states that
every algebraic group scheme over a field of characteristic zero
is smooth;   see  \cite[\S\,3.g, p.~71 Theorem~3.23]{Mil} or \cite[\S\,8.7.3, p.~236]{Pr07}, or~\cite{He12}.

\begin{proof}
We need to show that a commutative Hopf $\mathbin{\widehat{\otimes}}$-algebra $H$ is HFG if and only if there is a Stein group $G$ such that $\mathcal{O}(G)\cong H$. To prove the sufficiency note that for  a Stein group $G$  its underlying complex manifold has finite embedding dimension. Therefore, by \cite[Theorem 3.22]{Pi15}, the algebra $\mathcal{O}(G)$ is HFG.

The proof of the necessity splits into three steps. We show that: (1)~the spectrum of a commutative  HFG Hopf algebra is a group; (2)~it is a reduced Stein space and (3)~it is regular, i.e., a manifold. Below we use auxiliary results contained in Appendix~\ref{ap:top}.

(1)~The functor $(X,\mathcal{O}_X) \mapsto \mathcal{O}(X)$ is an anti-equivalence between the
category of Stein spaces of finite embedding dimension and the category of
commutative HFG algebras \cite[Theorems~3.22 and~3.23]{Pi15}. (Here we denote by $\mathcal{O}(X)$ the algebra of global sections.) Therefore for any commutative  HFG Hopf algebra~$H$ there exists a Stein space $X$ such that
$\mathcal{O}(X)\cong H$  as an HFG algebra. Since $\mathcal{O}(X)\mathbin{\widehat{\otimes}} \mathcal{O}(X)\cong
\mathcal{O}(X\times X )$ by Proposition~\ref{OOtiso}, it follows that the comultiplication $\mathcal{O}(X)\to\mathcal{O}(X)\mathbin{\widehat{\otimes}}
\mathcal{O}(X)$ induces a holomorphic map $X\times X\to X$.   Since $H$~is
commutative, the antipode is a homomorphism and hence  induces a holomorphic map $X\to X$. Also, the counit gives a distinguished element~$e$ of~$X$. It is not hard to check that the group axioms holds for $X$.

(2)~Further, we want to show that $(X,\mathcal{O}_X)$ is reduced, i.e., all nilpotent elements in the stalk
$$
\mathcal{O}_{X,x}\!:=\varinjlim_{x\in U}\mathcal{O}_X(U)
$$
are trivial for each $x\in X$.
By Proposition~\ref{indcattopa}, the colimit can be taken in the category of commutative $\mathbin{\widehat{\otimes}}$-algebras. Note that the functor $(-)\mathbin{\widehat{\otimes}} (-)$ is the coproduct functor in this category. So we can use the interchange property of colimits
for small categories:
\begin{multline} \label{OXXxy}
\mathcal{O}_{X\times X,(x,y)}=\varinjlim_{x\in U,\,y\in V}\mathcal{O}_{X\times X}(U\times V)\cong
\varinjlim_{x\in U,\,y\in V}\mathcal{O}_X(U)\mathbin{\widehat{\otimes}}
\mathcal{O}_X(V) \cong \\
\varinjlim_{x\in U}\mathcal{O}_X(U)\mathbin{\widehat{\otimes}} \varinjlim_{y\in V}\mathcal{O}_X(V) \cong\mathcal{O}_{X,x}\mathbin{\widehat{\otimes}}
\mathcal{O}_{X,y} \qquad (x,y\in X).
\end{multline}
(Note that $\mathcal{O}_{X\times X,(x,y)}$ can be identified with a coproduct in the category of local analytic algebras \cite[Chapter~3, \S\,5]{GR1} but the functional-analytic point of view seems more natural here.)

We claim that $\mathcal{O}_{X,e}$ is a Hopf $\mathbin{\widehat{\otimes}}$-algebra. Indeed, the multiplication  $X\times X\to X$ induces a morphism of locally ringed spaces 
$$ 
(X\times X ,\mathcal{O}_{X\times X})\to (X,\mathcal{O}_X)
$$ 
that satisfies the associativity property. In view of \eqref{OXXxy}
this morphism provides a coassociative homomorphism
$$
\De_e\!:\mathcal{O}_{X,e}\to\mathcal{O}_{X\times X,(e,e)}\cong\mathcal{O}_{X,e}\mathbin{\widehat{\otimes}} \mathcal{O}_{X,e}.
$$
We get an antipode and a counit similarly. It is not hard to see that the axioms of
Hopf $\mathbin{\widehat{\otimes}}$-algebra are satisfied.

Further, we proceed as in the pure algebraic case (cf. \cite[\S~8.7.3,
Theorem~1]{Pr07}).  Let $\mathfrak{m}$ denote the unique maximal ideal of $\mathcal{O}_{X,e}$.  Consider  the associated graded algebra $C\!:=\bigoplus_{k\in\mathbb{Z}_+} \mathfrak{m}^k/\mathfrak{m}^{k+1}$  (see, e.g., \cite[\S\,5.1]{Ei95}).

We claim that $C$ is a commutative graded Hopf algebra (in the standard sense, no topology is assumed). Indeed, being the image of the maximal ideal of an algebra of the convergent power series, $\mathfrak{m}$  is finitely generated. Then each power of it has finite codimension.  So for every $k,l\in \mathbb{Z}_+$ we can treat $\mathfrak{m}^k{\widehat{\otimes}}\mathfrak{m}^l$ as a subspace of $\mathcal{O}_{X,e}\mathbin{\widehat{\otimes}} \mathcal{O}_{X,e}$. Since the comultiplication  $\De_e\!:\mathcal{O}_{X,e}\to\mathcal{O}_{X,e}\mathbin{\widehat{\otimes}} \mathcal{O}_{X,e}$ is a homomorphism of local algebras, i.e., the image of the maximal ideal is contained in the maximal
ideal, we have $\De_e(\mathfrak{m})\subset \mathfrak{m}\mathbin{\widehat{\otimes}} \mathcal{O}_{X,e}+\mathcal{O}_{X,e}\mathbin{\widehat{\otimes}}  \mathfrak{m}$. The antipode $S_e\!:\mathcal{O}_{X,e}\to\mathcal{O}_{X,e}$ is also a homomorphism of
local algebras and hence  $S_e(\mathfrak{m})\subset \mathfrak{m}$. It clear that the counit vanishes on~$\mathfrak{m}$. Thus $\mathfrak{m}$ is a Hopf ideal. Then, as in the algebraic case,  $\De_e$, $S_e$ and the counit induce a structure of
commutative graded Hopf algebra on~$C$.

We now apply the Milnor-Moore theorem
\cite[Theorem~3.8.3]{Car}, which asserts that, in  characteristic~$0$,  a commutative graded Hopf
algebra over a field~$K$ with  zero term equal to $K$ is a free commutative algebra (a polynomial algebra).  Therefore $C$ contains no non-zero nilpotent element.

Now suppose that $r$  is a non-trivial nilpotent element in $\mathcal{O}_{X,e}$.
Being a local analytic algebra, $\mathcal{O}_{X,e}$ is Noetherian
\cite[Chapter~2, Theorem~1]{GR1}. So by the Krull intersection lemma \cite[Appendix, \S~2.3]{GR1}, we have that $\bigcap_{k\in\mathbb{Z}_+}\mathfrak{m}^k=\{0\}$.  Therefore  there is $k$ such that
$r$ is in $\mathfrak{m}^k$ but not in $\mathfrak{m}^{k+1}$. Since $r$ is nilpotent,  the
corresponding element  $r+\mathfrak{m}^{k+1}$ in $C$ is also nilpotent.
This contradiction  implies that $\mathcal{O}_{X,e}$ is reduced.

Denote  by $\de_x$ the evaluation homomorphism $H\to \mathbb{C}$ at $x\in X$.
Then $(\de_x \otimes 1)\De$ and  $(\de_{x^{-1}} \otimes 1)\De$  are mutually
inverse continuous homomorphisms of Stein algebras.  By Forster's theorem \cite{Fo67},
we have a Stein space automorphism of~$X$, which maps $e$ to $x$.
Therefore the stalks  $\mathcal{O}_{X,x}$ and $\mathcal{O}_{X,e}$ are isomorphic. Thus,  $X$
is reduced at each point.

(3)~Finally, we claim that $(X,\mathcal{O}_X)$ is regular, i.e., it contains no
singular points. Suppose the contrary: let~$x$ be a singular point of~$X$. Then for every $y\in X$ we have
$\mathcal{O}_{X,x}\cong\mathcal{O}_{X,y}$ and so~$y$  is also singular. To get a
contradiction note that the set of singular points of a reduced complex
analytic space is nowhere dense \cite[Chapter~6, \S2.2]{GR2}.

Thus, $X$ is a Stein manifold and a group with holomorphic multiplication and holomorphic inverse, i.e., it is a Stein group.
 \end{proof}

\begin{rem}
Theorem~\ref{antiSTcoH} implies, in particular, that if $G$ is a  complex Lie group that is not a Stein group, then the Hopf $\mathbin{\widehat{\otimes}}$-algebra $\mathcal{O}(G)$  is isomorphic to $\mathcal{O}(G_S)$ for some Stein group $G_S$. When $G$ is connected this fact can be also derived from Morimoto's result, which asserts that for every  connected  complex  Lie group  $G$  there  exists  the smallest closed complex subgroup $G^0$  such that  $G/G^0$  is a Stein group \cite[Theorem~1]{Mo65}. (In \cite{ArLi} it is proposed to call $G^0$ the \emph{Morimoto subgroup} of~$G$.)
 \end{rem}

We finish this section with examples of  commutative Fr\'echet Arens--Michael Hopf algebras that are not HFG. Note that the algebras $\mathfrak{A}_s$ considered here will also be used in the proof of Theorem~\ref{AMEAG}.

\begin{exm}\label{commnHFG}
For $s\ge 0$ put
\begin{equation}
 \label{faAsdef}
\mathfrak{A}_s\!:=\Bigl\{a=\sum_{n=0}^\infty  a_n x^n\! :
\|a\|_{r,s}\!:=\sum_{n=0}^\infty |a_n|\frac{r^n}{n!^s}<\infty
\;\forall r>0\Bigr\},
\end{equation}
where $x$ is a formal variable.
It is not hard to see that $\mathfrak{A}_s$ is a Fr\'echet Arens--Michael algebra   with respect to the multiplication extended from the polynomial algebra  $\mathbb{C}[x]$ \cite[Proposition~4]{ArRC}.

We claim that other Hopf algebra operations can also  be extended from $\mathbb{C}[x]$ to $\mathfrak{A}_s$ and, moreover,  $\mathfrak{A}_s$ is a Hopf $\mathbin{\widehat{\otimes}}$-algebra. The only non-trivial assertion is that the
comultiplication is continuous.
It is a standard fact from the K\"{o}the space theory (see, for example, \cite[Proposition~3.3]{Pi02}) that
\begin{equation}\label{Kothetpr}
\mathfrak{A}_s\mathbin{\widehat{\otimes}} \mathfrak{A}_s\cong
\Bigl\{c=\sum_{i,j=0}^\infty  c_{ij} x^i\otimes x^j\! :
\|c\|'_{r,s}\!:=\sum_{i,j=0}^\infty |c_{ij}|\frac{r^{i+j}}{(i!j!)^s}<\infty
\;\forall r>0\Bigr\}\,.
\end{equation}
Since $\De(x)=x\otimes 1+1\otimes x$, we have for any $p=\sum_n a_nx^n\in \mathbb{C}[x]$ that
$$
\De(p)=\sum_n a_n \sum_{i+j=n} {n \choose j}x^i\otimes x^j.
$$
Therefore,
$$
\|\De(p)\|'_{r,s}=\sum_n \, \frac{|a_n|\,r^n}{n!^{s}}\, \sum_{i+j=n} {n \choose j}^{s+1}.
$$
Since $s\ge 0$, we have for every $n$ that
$$
\sum_{i+j=n}   {n \choose j}^{s+1}\le \left(\sum_{i+j=n}   {n \choose j}\right)^{s+1}=2^{n(s+1)}.
$$
Hence,
$$
\|\De(p)\|'_{r,s}\le\sum_n |a_n| \, \frac{2^{n(s+1)}\,r^n}{n!^{s}}
$$
and so $\De$ is continuous.

Thus, $\mathfrak{A}_s$ is a Fr\'echet Arens--Michael Hopf algebra, which is obviously commutative and cocommutative. But it is not HFG when $s>0$  by \cite[Proposition~10]{ArRC}. These assertions hold also for $\mathfrak{A}_\infty\!:=\mathbb{C}[[x]]$ (the algebra of formal power series).
 \end{exm}

 \section{Cocommutative  HFG Hopf algebras\\  of the form $\widehat {\mathscr A}(G)$: General results}
\label{cocHHFGamt}

 We discuss properties of the Arens--Michael envelope $\widehat{\mathscr A}(G)$ of the algebra ${\mathscr A}(G)$ of analytic functionals for a complex Lie group~$G$. Since ${\mathscr A}(G)$ is dual to the commutative Hopf HFG algebra $\mathcal{O}(G)$, which is a nuclear Fr\'echet space, it is a cocommutative Hopf $\mathbin{\widehat{\otimes}}$-algebra and so is $\widehat{\mathscr A}(G)$.

Recall that a locally compact group~$G$ is said to be
\emph{compactly generated} if there is a relatively compact
symmetric  set that generates~$G$.
The following theorem is the main result in this section.

 \begin{thm}\label{cicHFGgen}
Let $G$ be a complex Lie group countable at infinity. The following conditions are equivalent:
\begin{mycompactenum}
\item
$\widehat{\mathscr A}(G)$ is an HFG algebra;
\item
$\widehat{\mathscr A}(G)$ is a Fr\'echet algebra;
\item $G$ is compactly generated.
\end{mycompactenum}
\end{thm}

See Example~\ref{commnHFG} for cocommutative Hopf $\mathbin{\widehat{\otimes}}$-algebras that are not HFG.

We begin the proof of Theorem~\ref{cicHFGgen} with the implication
$\mathrm{(iii)}\Longrightarrow(\mathrm{i)}$. The argument will
proceed from a partial case to general in three steps:

1.~for simply-connected $G$;

2.~for arbitrary  connected $G$;

3.~for arbitrary  compactly generated $G$.

(Note that  a connected locally compact group is always compactly generated \cite[Theorem~5.7]{HeRo}.)

\subsection*{Step 1:  $G$ is simply connected} In fact, this case is considered in \cite{Pir_stbflat}.
If  $\mathfrak{g}$ is the Lie algebra of  a complex Lie group~$G$, then there is a
homomorphism of  Hopf $\mathbin{\widehat{\otimes}}$-algebras $\tau\!:U(\mathfrak{g}) \to {\mathscr A}(G)$ (see, e.g.,
\cite[(42)]{Pir_stbflat}). Since the  Arens--Michael envelope is an
endofunctor on the category of Hopf $\mathbin{\widehat{\otimes}}$-algebras [ibid. Proposition~6.7], we have the natural
homomorphism  $\widehat\tau\!:\widehat U(\mathfrak{g}) \to \widehat{\mathscr A}(G)$  of Hopf $\mathbin{\widehat{\otimes}}$-algebras.

\begin{pr}\label{csicHFG}
Let $G$ be a  simply-connected complex Lie group with Lie algebra~$\mathfrak{g}$.
Then the homomorphism $$\widehat\tau\!:\widehat U(\mathfrak{g}) \to
\widehat{\mathscr A}(G)$$ of $\mathbin{\widehat{\otimes}}$-Hopf algebras  is an  isomorphism.  As a corollary, $\widehat{\mathscr A}(G)$ is an HFG  Hopf
algebra.
\end{pr}
\begin{proof}
The first assertion follows form  \cite[Proposition~9.1]{Pir_stbflat} as explained in
\cite[Proposition~2.1]{ArAMN}. Since $U(\mathfrak{g})$ is affine, $\widehat{\mathscr A}(G)\cong \widehat U(\mathfrak{g})$ is an HFG algebra
by Proposition~\ref{ANenvaff}.
\end{proof}

In Steps~2 and~3 we use auxiliary results contained in Appendices.

\subsection*{Step 2: $G$ is  connected}

\begin{pr}\label{cicHFG}
Let $G$ be a connected complex Lie group. Then  $\widehat{\mathscr A}(G)$ is an HFG  Hopf
algebra.
\end{pr}
For the proof we need a lemma.
\begin{lm}\label{Ptba0}
Let $\pi\!:A\to B$ be an open surjective homomorphism of
$\mathbin{\widehat{\otimes}}$-algebras. If $\widehat A$ is a Fr\'{e}chet algebra, so is $\widehat
B$. If $\widehat A$ is an HFG algebra, so is $\widehat B$.
\end{lm}
\begin{proof}
Evidently, $\Ker \pi$ is a closed two-sided ideal in  $A$. Since
$\pi$ is an open surjection,  we have $A/\Ker \pi \cong B$. If~$\widehat A$ is a Fr\'{e}chet algebra, then it follows from \cite[Corollary~6.2]{Pir_stbflat}  that $\widehat B\cong \widehat A/I$, where $I$ is the closure of the
image of $\Ker \pi$ in $\widehat A$. Finally, note that a quotient of a
Fr\'{e}chet algebra is a Fr\'{e}chet algebra and a quotient of an
HFG algebra is~HFG.
\end{proof}

\begin{proof}[Proof of Proposition~\ref{cicHFG}]
Let $G$ be a connected complex Lie group, $\widetilde G$  its universal covering and
$\pi\!:\widetilde G\to G$  the quotient homomorphism. Then
$\widetilde G$ is a complex Lie group and $\pi$ is a Lie group
homomorphism \cite[Chapitre~III, \S~1, no.~9]{Bo06}. So there is a
closed normal subgroup $N$ in $\widetilde G$ such that $G\cong\widetilde G/N$. Since $\widetilde G$ is connected, it is countable at infinity. By Proposition~\ref{opensurvl}, the corresponding $\mathbin{\widehat{\otimes}}$-algebra
homomorphism $\varphi'\!:{\mathscr A}(\widetilde G)\to{\mathscr A}(G)$ is
surjective and open.  Since $\widetilde G$ is simply connected, it follows
from  Proposition~\ref{csicHFG} that $\widehat{\mathscr A}(\widetilde G)$ is an
HFG algebra. Finally, we apply Lemma~\ref{Ptba0} to $\varphi'$.
\end{proof}

\subsection*{Step 3:  $G$ is compactly generated}

Let $G$ be a compactly generated complex Lie group and  $G_0$ denote the  component of identity. Then $\Ga\!:=G/G_0$ is a finitely generated discrete group (because $G_0$ is open).  We fix generators $\ga_1,\ldots, \ga_m$ of
$\Ga$  and their preimages $g_1,\ldots, g_m$ under $G\to \Ga$. Let~$\Ga_1$ be the subgroup in~$G$ generated by  $g_1,\ldots, g_m$.

For $g\in G$,  denote by $\de_g$ the corresponding delta function;
obviously, $\de_g\in {{\mathscr A}}(G)$. The multiplication in ${{\mathscr A}}(G)$ is denoted
by $\star$. The support $\mathop{\mathrm{supp}}\nu$ of $\nu\in {{\mathscr A}}(G)$ is defined in the usual way.

\begin{lm}\label{sCgA}
For each $\nu\in {{\mathscr A}}(G)$ there is a finite subset $S$ in $\Ga_1$ such that
$$
\nu=\sum_{g\in S} \de_{g}\star \nu_g ,
$$
where $\mathop{\mathrm{supp}}\nu_g\subset G_0$.
\end{lm}
\begin{proof}
Each coset of $G_0$ is an open subset in $G$; hence it is a complex manifold.
So we can write $\mathcal{O}(G)\cong\prod_{\ga\in \Ga} \mathcal{O}(\ga)$. Since the strong
dual of a product is topologically isomorphic to the direct sum of the strong
duals \cite[\S 22.5]{Kot1}, we have
$$
{{\mathscr A}}(G)=\mathcal{O}(G)'\cong\left(\prod_{\ga\in \Ga}
\mathcal{O}(\ga)\right)'\cong \bigoplus_{\ga\in \Ga}{{\mathscr A}}(\ga)\,.
$$
Hence, for each $\nu\in {{\mathscr A}}(G)$  there are  a finite subset $S_0$  in
$\Ga$  and  $\{\eta_\ga \in {{\mathscr A}}(G):\,\ga\in S_0\}$  with $\mathop{\mathrm{supp}}\eta_\ga\subset \ga$ such that $\nu=\sum_{\ga\in
S_0} \eta_\ga$. For each
$\ga\in S_0$ fix $g_\ga\in\Ga_1\cap \ga$, denote the set of all such~$g_\ga$ by~$S$ and put $\nu_{g_\ga}\!:=\de_{g_\ga^{-1}}\star\eta_\ga$. Hence $\mathop{\mathrm{supp}} \nu_{g_\ga}\subset G_0$ and the proof is complete.
\end{proof}

For a discrete semigroup $S$ we denote by  $\mathbb{C} S$ the semigroup algebra of~$S$
with respect to the convolution multiplication. Note that $\mathbb{C} S$ is a $\mathbin{\widehat{\otimes}}$-algebra with respect to the strongest
locally convex topology. Topologically, it
is a direct sum of copies of $\mathbb{C}$. So if $S$ is countable, then  $\mathbb{C} S$ is
a nuclear (DF)-space. (Note that when $S$ is locally finite, the Arens--Michael envelope $\widehat{\mathbb{C} S}$  coincides with $\mathbb{C} S$ and so it is also a nuclear (DF)-space \cite{Ar22} but we do not need this fact here.)

\begin{proof}[The proof of implication $\mathrm{(iii)}\Longrightarrow(\mathrm{i)}$ in Theorem~\ref{cicHFGgen}]
Let $G_0$ be the component of identity. Consider the finitely-generated discrete group $\Ga_1$ chosen above. The holomorphic homomorphisms $\Ga_1\to G$ and $G_0\to G$ induce the continuous $\mathbin{\widehat{\otimes}}$-algebra homomorphisms $\mathbb{C} \Ga_1\to {{\mathscr A}}(G)$ and   ${{\mathscr A}}(G_0)\to {{\mathscr A}}(G)$.  By the universal property of the free product of $\mathbin{\widehat{\otimes}}$-algebras (see Appendix~\ref{ap:free}), we have the continuous homomorphism
$$
\pi\!:\mathbb{C} \Ga_1\mathbin{\ast}{{\mathscr A}}(G_0)\to {{\mathscr A}}(G)\,.
$$
By Lemma~\ref{Ptba0}, to prove  that  $\widehat{{\mathscr A}}(G)$ is HFG it suffices to show that $(\mathbb{C} \Ga_1\ast {{\mathscr A}}(G_0))\sphat\,\,$ is an HFG algebra and that~$\pi$ is surjective and open.

It follows from Proposition~\ref{frpris} that $(\mathbb{C} \Ga_1\ast {{\mathscr A}}(G_0))\sphat\,\,$ is isomorphic to $\widehat{\mathbb{C}\Ga}_1\mathbin{\widehat\ast} \widehat{{\mathscr A}}(G_0)$, where ${\widehat\ast}$ stands for the free product of Arens--Michael algebras. Since $\mathbb{C} \Ga_1$ is a finitely
generated algebra, $\widehat{\mathbb{C} \Ga}_1$ is HFG \cite[Proposition~7.1]{Pi15}.  On the other hand, $G_0$ is connected, so
$\widehat{{\mathscr A}}(G_0)$ is HFG by Proposition~\ref{cicHFG}. Since the Arens--Michael
free product of two HFG algebras is HFG \cite[Corollary~4.7]{Pi15}, we have that $(\mathbb{C} \Ga_1\ast {{\mathscr A}}(G_0))\sphat\,\,$ is~HFG.

On the other hand,  note that $\mathbb{C}^{\Ga_1}$ and $\mathcal{O}(G_0)$  are nuclear Fr\'{e}chet space (the latter follows from Lemma~\ref{Ptbam}). Hence the strong dual spaces  $\mathbb{C} \Ga_1$  and ${{\mathscr A}}(G_0)$  are nuclear (DF)-spaces by \cite[p.\,84, Theorem 4.4.13]{Pie}. Then so is   $\mathbb{C} \Ga_1\ast
{{\mathscr A}}(G_0)$  by Proposition~\ref{ga1APt}.  Being a strong dual of a reflexive
Fr\'{e}chet space, it is a Pt\'{a}k space  \cite[\S\,34.3(5)]{Kot2}.
It follows from Lemma~\ref{sCgA}  that $\pi$ is surjective.
Proposition~\ref{Ptba}  implies that  ${{\mathscr A}}(G)$  is  barreled. So, by  Pt\'{a}k's
open mapping theorem, $\pi$ is open  (see the references in the proof of Proposition~\ref{opensurvl}).
This completes the proof.
\end{proof}

Now we turn to the implication $\mathrm{(ii)}\Longrightarrow(\mathrm{iii)}$ in
Theorem~\ref{cicHFGgen}.

The case of a countable discrete group is basic. In the discussion below it is convenient to consider groups as semigroups (even not necessarily with an identity). So we need a description of $\widehat{\mathbb{C} S}$ for a countable discrete semigroup~$S$.
(Note that $\mathbb{C} S$ is not unital unless~$S$ has an identity.)

Recall that a  \emph{length function} on a locally compact semigroup
$S$ is a locally bounded function $\ell\!:S\to \mathbb{R}$ such that
$$
 \ell(st)\le \ell(s)+\ell(t)\qquad (s, t \in S)\,.
$$

Suppose that $S$ is a countable discrete semigroup with generating subset $\{s_n\!:n\in\mathbb{N}\}$. The following construction is known in the abelian case; see, e.g.,  \cite{Pl13}.  For any function $F\!:\mathbb{N}\to \mathbb{R}_+$ put
\begin{equation}\label{locwordlen}
\ell_F(s)\!: = \inf\left\{\sum_{k=1}^m F(n_k)\!: \,
s=s_{n_1}\cdots s_{n_m}\!:\, n_1,\ldots, n_m \in
\mathbb{N}\right\}.
\end{equation}
(Here as above $\de_s$ denotes the delta function at $s$.)
It is easy to see that $\ell_F $ is a non-negative length function.
We write any element~$a$ of $\mathbb{C} S$ as a finite sum $\sum_{s\in S}a_s\de_s$
and put
$$
\|a\|_F\!:=\sum_{s\in S}|a_s|\,e^{\ell_F(s)}.
$$
 Obviously,  $\|\cdot\|_F$  is a submultiplicative seminorm on $\mathbb{C} S$. Denote by $(\mathbb{C} S)_F$ the completion of $\mathbb{C} S$ with respect to $\|\cdot\|_F$. Then we have an algebra homomorphism $\mathbb{C} S\to (\mathbb{C} S)_F$.

We order non-negative functions on $\mathbb{N}$  by increasing. It is easy to see that for $F_1\le F_2$ there is a continuous homomorphism $(\mathbb{C} S)_{F_2}\to (\mathbb{C} S)_{F_1}$ of Banach algebras. So we can consider a directed projective system $((\mathbb{C} S)_F)$ of unital Banach
algebras, where~$F$ runs over all possible non-negative functions on~$\mathbb{N}$.

\begin{lm}\label{prAga}
If $S$ is a countable discrete semigroup with  generating subset $\{s_n\!:n\in\mathbb{N}\}$, then
the $\mathbin{\widehat{\otimes}}$-algebra homomorphism
$$
\mathbb{C} S\to \varprojlim((\mathbb{C} S)_F\!:\, F\in \mathbb{R}_+^\mathbb{N})
$$
is an Arens--Michael envelope.
\end{lm}
(Here we consider the projective limit in the category of locally convex spaces, which automatically the projective limit in the category of Arens--Michael algebras.)
\begin{proof}
It suffices to show that any submultiplicative seminorm~$\|\cdot\|$ on~$\mathbb{C} S$ can be majorized by a seminorm of the form $\|\cdot\|_F$.

Fix $\|\cdot\|$  and put $F(n)\!:=|\log\|\de_{s_n}\|\,|$.
If $s=s_{n_1}\cdots s_{n_m}$ as in~\eqref{locwordlen}, then
$$
\|\de_s\|\le\|\de_{s_{n_1}}\|\cdots \|\de_{s_{n_m}}\|\le\exp\left(\sum_k |\log\|\de_{s_{n_k}}\|\,|\right) \,.
$$
Therefore $\|\de_s\|\le e^{\ell_F(s)}$ for
each  $s\in S$ and so $\|a\|\le \|a\|_F$ for every
$a\in\mathbb{C} S$.
\end{proof}

For a fixed generating subset $\{s_n\!:n\in\mathbb{N}\}$ of~$S$
denote by $S_n$ the subsemigroup generated by $s_1,\ldots,s_n$.

\begin{lm}\label{arbtCnF}
If $s_n\notin S_{n-1}$ for each $n\in\mathbb{N}$, then
for any sequence $(C_n)$ in~$\mathbb{R}_+$ there is a non-negative function $F$ such that $\ell_F(s_n)>C_n$ for each $n\in\mathbb{N}$.
\end{lm}
\begin{proof}
Take $F$ such that $F(n)>\max\{C_1,\ldots, C_n\}$. If $n\in\mathbb{N}$, then for any decomposition $s_n=s_{n_1}\cdots s_{n_m}$ there is $j$ such that $s_{n_j}\notin S_{n-1}$. So we have from~\eqref{locwordlen} that
$$
\ell(s_n)\ge \min\{F(m)\!:\, m\ge n\}>C_n.
$$
\end{proof}

\begin{pr}\label{Frfgen}
Let $S$ be a countable discrete semigroup. Then
$\widehat{\mathbb{C} S}$ is a Fr\'echet space if and only if~$S$ is finitely generated.
\end{pr}
\begin{proof}
The sufficiency is straightforward; cf. \cite[Proposition~7.1]{Pi15}.

To prove the necessity suppose that $\widehat{\mathbb{C} S}$ is a Fr\'echet space and
assume  to the contrary that~$S$~is not finitely generated.
For a sequence $\{s_n\!:n\in\mathbb{N}\}$ in~$S$ we denote as above by $S_n$  the subsemigroup generated by
$s_1,\ldots,s_n$. It is easy to see by induction that   $\{s_n\!:n\in\mathbb{N}\}$ can be chosen with $s_n\not\in S_{n-1}$ for each $n\in\mathbb{N}$.

By Lemma~\ref{prAga}, the  system $(\|\cdot\|_F)$, where~$F$ runs all non-negative functions on~$\mathbb{N}$, is defining
for $\widehat{\mathbb{C} S}$.  Since $\widehat{\mathbb{C} S}$ is a Fr\'echet space,
there is an increasing sequence $(F_n)$ of functions such that
$(\|\cdot\|_{F_n})$ is a defining system.

Now we use Cantor's diagonal argument to get a contradiction.
It follows from Lemma~\ref{arbtCnF} that there is
a function~$F$ such that
\begin{equation}\label{Fnmax}
\ell_F(s_n)> n+\ell_{F_n}(s_n)
\end{equation}
for every $n$.
Since $(\|\cdot\|_{F_n})$ is a defining system and the sequence $(F_n)$ is increasing, there are $k$ and $C>0$ such that $\|a\|_F\le C\|a\|_{F_ n}$ for each $a\in\mathbb{C} S$ and $n\ge k$. Putting $a=\de_s$, we have $\ell_F(s)\le \ell_{F_n}(s) +\log C$ for every $s\in S$.  Combining this with \eqref{Fnmax},
we conclude that $\log C>n$ when~$n\ge k$ and get a contradiction, which completes the proof.
\end{proof}

Thus, for a finitely generated group $G$, to determine the topology on the Arens--Michael envelope it suffices to take an unbounded countable set of length functions or even the set of integer multiples of a word length function.
The algebra defined in this way is considered in \cite{Me04,Me06}, where it is denoted by $\mathscr{O}(G)$. By Lemma~\ref{prAga}, this algebra  coincides with $\widehat{\mathbb{C} G}$.

\begin{proof}[The rest of the proof of Theorem~\ref{cicHFGgen}]
\,

$\mathrm{(ii)}\Longrightarrow(\mathrm{iii)}$.
Suppose that $\widehat{\mathscr A}(G)$ is a Fr\'echet space. Denote as above the component of $G$ by $G_0$ and the quotient group $G/G_0$ by~$\Ga$.
Since $G$ is countable at infinity, Proposition~\ref{opensurvl}
implies that ${\mathscr A}(G)\to {\mathscr A}(\Ga)$ is surjective and open. So, by
Lemma~\ref{Ptba0}, $\widehat{\mathscr A}(\Ga)$ is a Fr\'echet space.
Recall that ${\mathscr A}(\Ga)=\mathbb{C}\Ga$. Since~$\Ga$ is countable, Proposition~\ref{Frfgen} implies that~$\Ga$ is finitely generated. Therefore~$G$ is compactly generated.

$\mathrm{(i)}\Longrightarrow(\mathrm{ii)}$. Any HFG algebra is a Fr\'echet algebra by definition.
\end{proof}

\subsection*{Cocommutative  Hopf HFG algebras that are not associated with Lie groups}

The natural question arises whether all cocommutative  HFG Hopf algebras are of the form $\widehat{\mathscr A}(G)$ for some complex Lie group~$G$. In this section we show that this is not the case. Examples come from finitely generated (not necessarily finite-dimensional)  complex Lie algebras. Indeed, if
$\mathfrak{g}$ is such a Lie algebra, then $U(\mathfrak{g})$ is finitely generated as an algebra, hence, an affine Hopf algebra. Thus $\widehat U(\mathfrak{g})$ is a cocommutative  Hopf HFG algebra by Proposition~\ref{ANenvaff}.

\begin{thm}\label{fgenLiea}
Let $\mathfrak{g}$ be a finitely generated complex Lie algebra and $G$ a complex Lie group countable at infinity. If there is a Hopf $\mathbin{\widehat{\otimes}}$-algebra isomorphism $\widehat U(\mathfrak{g})\to \widehat{\mathscr A}(G)$,  then   $\mathfrak{g}$ is  finite dimensional.
\end{thm}

We need two auxiliary results on primitive elements. As in the case of a Hopf algebra, we say that
an element~$X$ of a Hopf $\mathbin{\widehat{\otimes}}$-algebra~$H$ is \emph{primitive} if $\De(X)=1\otimes X+X\otimes 1$ and denote
the linear subspace of  primitive elements by~$P(H)$.

\begin{pr}\label{primfordis}
If $\Ga$ is a countable discrete group, then $P(\widehat{\mathbb{C}\Ga})=0$.
\end{pr}
\begin{proof}
Since  $\mathbb{C}\Ga\otimes\mathbb{C}\Ga \cong \mathbb{C}(\Ga\times \Ga)$, we have from the compatibility of the Arens--Michael enveloping functor with the projective tensor product \cite[Proposition~6.4]{Pir_stbflat} that
$$
\widehat{\mathbb{C}\Ga}\mathbin{\widehat{\otimes}} \widehat{\mathbb{C}\Ga}\cong \widehat{\mathbb{C}(\Ga\times\Ga)}.
$$

It follows from Lemma~\ref{prAga} that $\widehat{\mathbb{C}\Ga}$ is the K\"{o}the sequence space corresponding to the K\"{o}the set $P\!:=\{e^{\ell_F}\}$, where~$F$ runs over all possible non-negative functions on a fixed subset generating~$\Ga$ as a semigroup and  $\ell_F$ is defined in~\eqref{locwordlen}. Obviously, the same is true
for $\widehat{\mathbb{C}(\Ga\times\Ga)}$. In particular, each element of $\widehat{\mathbb{C}\Ga}$ ($\widehat{\mathbb{C}\Ga}\mathbin{\widehat{\otimes}} \widehat{\mathbb{C}\Ga}$, resp.) can be uniquely decomposed as a convergent series in the basis $\{\de_\ga\!:\ga\in\Ga\}$ ($\{\de_\ga\otimes \de_{\ga'}\!:\ga,\ga'\in\Ga\}$, resp.).

For $X\in\widehat{\mathbb{C}\Ga}$ write $X=\sum_{\ga\in\Ga} c_\ga \de_\ga$. Hence
$\De(X)=\sum_{\ga\in\Ga} c_\ga \de_\ga\otimes\de_\ga$ and
$$
1\otimes X+X\otimes 1=\sum_{\ga\in\Ga} c_\ga (\de_e\otimes\de_\ga+\de_\ga\otimes\de_e),
$$
where $e$ is the identity element. If $\De(X)=1\otimes X+X\otimes 1$, then $c_\ga=0$ for all $\ga$ and thus $X=0$.
\end{proof}

\begin{pr}\label{connprim}
If $G$ is a connected complex Lie group, then $P(\widehat{\mathscr A}(G))$ is finite dimensional.
\end{pr}
\begin{proof}
Since $G$ is connected, it is compactly generated and so $\widehat{\mathscr A}(G)$ can be identified with the Hopf $\mathbin{\widehat{\otimes}}$-algebra ${\mathscr A}_{exp}(G)$ of exponential analytic functionals. Furthermore, ${\mathscr A}_{exp}(G)$ is the strong dual of the Hopf $\mathbin{\widehat{\otimes}}$-algebra $\mathcal{O}_{exp}(G)$
of holomorphic functions of exponential type. For details see Section~\ref{Cocoexm} below.

If $X$ is a primitive element of ${\mathscr A}_{exp}(G)$, then
$$
\langle X, f_1f_2\rangle=\langle \De(X), f_1\otimes f_2\rangle= f_1(e) \langle X, f_2\rangle+f_2(e)\langle X, f_1\rangle
$$
for every $f_1,f_2\in \mathcal{O}_{exp}(G)$, i.e., $X$ belongs to $\mathop{\mathrm{Der}}(\mathcal{O}_{exp}(G),\mathbb{C})$, the set of continuous
point derivations of $\mathcal{O}_{exp}(G)$ at the identity. On the other hand, it is obvious that any continuous point derivation is a primitive element of~$\widehat{\mathscr A}(G)$. Hence $P(\widehat{\mathscr A}(G))=\mathop{\mathrm{Der}}(\mathcal{O}_{exp}(G),\mathbb{C})$.

Note that we can assume that  $G$ is linear \cite[Theorem~5.3(A)]{Ar19}. So $G$  can be considered as a non-singular affine variety. Then $\mathcal{O}_{exp}(G)$ contains the algebra ${\mathcal R}(G)$ of regular functions on~$G$  [ibid., Theorem~5.10] and the algebra homomorphism  ${\mathcal R}(G)\to\mathcal{O}_{exp}(G)$  has dense range [ibid., Corollary~5.11] (see a short proof in \cite{ArDP}). Therefore the induced linear map
$$
\mathop{\mathrm{Der}}(\mathcal{O}_{exp}(G),\mathbb{C})\to\mathop{\mathrm{Der}}({\mathcal R}(G),\mathbb{C})
$$
is injective. Here ${\mathcal R}(G)$ is endowed with
the strongest locally convex topology and so $\mathop{\mathrm{Der}}({\mathcal R}(G),\mathbb{C})$ is just the
space of all point derivations  at the identity. Since~$G$  contains no singular point, the latter space can be identified with the tangent space at identity, which is finite dimensional.
\end{proof}

\begin{proof}[Proof of Theorem~\ref{fgenLiea}]
Since $G$  is countable at infinity, $\Ga\!:=G/G_0$ is countable. Then, by
Proposition~\ref{opensurvl}, the homomorphism ${\mathscr A}(G)\to {\mathscr A}(\Ga)$ is surjective. Therefore the induced homomorphism $\widehat{\mathscr A}(G)\to \widehat{\mathscr A}(\Ga)$ has dense range.\footnote{In fact, it is surjective but we do not need this fact.}

Recall that by the assumption, $\widehat U(\mathfrak{g})\cong \widehat{\mathscr A}(G)$. Note that the subalgebra of $\widehat U(\mathfrak{g})$  generated by primitive elements is dense and a Hopf $\mathbin{\widehat{\otimes}}$-algebra homomorphism maps primitive elements to primitive. Therefore the subalgebra of $\widehat{\mathscr A}(\Ga)$ generated by primitive elements is dense. It follows from Proposition~\ref{primfordis} that $\Ga$ is trivial; hence $G$ is connected.  So we can apply Proposition~\ref{connprim}, which implies that $P(\widehat{\mathscr A}(G))$ is finite dimensional.
Since each element of $\mathfrak{g}$ is primitive in $\widehat U(\mathfrak{g})$, we get that $\mathfrak{g}$ is finite dimensional also.
\end{proof}

\begin{exm}\label{freeLa}
Let $\mathfrak{f}_n$ be a free Lie algebra in~$n$ generators with $n>1$. Since $\mathfrak{f}_n$ is finitely generated but not finite dimensional,  Theorem~\ref{fgenLiea} implies that the  HFG Hopf algebra $\widehat{U}(\mathfrak{f}_n)$ cannot be represented as $\widehat{\mathscr A}(G)$, where~$G$ is a complex Lie group countable at infinity.
\end{exm}

 \section{Cocommutative  HFG Hopf algebras\\ of the form $\widehat {\mathscr A}(G)$: Examples}
\label{Cocoexm}
In this section we consider  examples of a cocommutative  HFG Hopf algebra $\widehat{\mathscr A}(G)$ associated with a compactly generated complex Lie group~$G$ and describe the structure of $\widehat {\mathscr A}(G)$ in the case when~$G$ is connected.

 \subsection*{Non-connected groups}
We begin with two examples in which~$G$ is not connected and consider first a discrete group. While the description of $\widehat{\mathscr A}(\Ga)$ for a countable discrete group~$\Ga$ in the power series  form is available (see Lemma~\ref{prAga}), a representation its elements as matrix-valued functions is interesting from point of view of the noncommutative geometry.

 \begin{exm}
Consider the free product $\Ga\:=\mathbb{Z}_2\ast\mathbb{Z}_2$, i.e., $\Ga$ has generators~$u$ and~$v$ of order~$2$. Note that $\Ga$~is also generated by $u$ and $w\!:=uv$ and hence is isomorphic to $\mathbb{Z}_2\ltimes\mathbb{Z}$, where~$w$ generates~$\mathbb{Z}$ and~$\mathbb{Z}_2$ acts on~$\mathbb{Z}$ by $uwu=w^{-1}$.
Here we describe the  Hopf HFG algebra $\widehat{\mathbb{C}\Ga}$ as an algebra of matrix-valued holomorphic functions.

Since each element of  $\Ga$ has the forms $w^n$ or $uw^n$ ($n\in\mathbb{Z}$),
an arbitrary element of  $\mathbb{C}\Ga$ can be written as a finite sum
$$
a=\sum_{n\in\mathbb{Z}} (a_n w^n  +b_n uw^n),
$$
where $a_n,b_n\in\mathbb{C}$.

Consider the homomorphism
$$
\varphi\!:\mathbb{C}\Ga\to \mathrm{M}_2(\mathbb{C}[z,z^{-1}])\!:u\mapsto \begin{pmatrix}0 &
1\\1& 0
\end{pmatrix},\quad w\mapsto\begin{pmatrix}z & 0\\0& z^{-1}
\end{pmatrix}.
$$
Since for any $\rho>0$ the formula
$$
|\sum_{n\in\mathbb{Z}}a_nz^n|_\rho\!:= \sum_{n\in\mathbb{Z}}|a_n|\rho^{|n|}
$$
defines a submultiplicative seminorm on $\mathbb{C}[z,z^{-1}]$, it follows that
$$
\begin{pmatrix}f_{11} & f_{12} \\
f_{21} &  f_{22}
\end{pmatrix}  \mapsto   |f_{11}|_\rho+| f_{12}|_\rho+|f_{21} |_\rho+|f_{22} |_\rho
$$
is a submultiplicative seminorm on  $\mathrm{M}_2(\mathbb{C}[z,z^{-1}])$.\footnote{Here the seminorm of the identity is not equal to~$1$; it does not impair the argument that follows.} Denote
by $\|\cdot\|_\rho$ its restriction on~$\mathbb{C}\Ga$.
Note that for any $a\in\mathbb{C}\Ga$ we have
$$
\varphi(a)=\sum_{n\in\mathbb{Z}}
\begin{pmatrix}a_nz^n & b_nz^{-n}\\
b_nz^{n}&   a_nz^{-n}
\end{pmatrix};
$$
therefore  $\|f\|_\rho=2\sum_{n\in\mathbb{Z}}(|a_n|+|b_n|)\rho^{|n|} $.

It is easy to see that for any submultiplicative seminorm   $\|\cdot\|$ on $\mathbb{C}\Ga$ we have $\|a\|\le \|a\|_\rho/2$ for some $\rho>0$. Thus $\widehat{\mathbb{C}\Ga}$ is the completion   of the range of~$\varphi$ in
$\mathrm{M}_2(\mathcal{O}(\mathbb{C}^\times))$.

Consider the automorphism of $\mathrm{M}_2(\mathcal{O}(\mathbb{C}^\times))$ of order~$2$ defined by
$$
\si\!:\begin{pmatrix}f_{11} & f_{12} \\
f_{21} &  f_{22}
\end{pmatrix}\mapsto
\begin{pmatrix}\check f_{22} &\check f_{21} \\
\check f_{12} &\check  f_{11}
\end{pmatrix},
$$
where $\check f(z)=f(z^{-1})$. So we finally have  that
$$
\widehat{\mathbb{C}\Ga}\cong\{b\in \mathrm{M}_2(\mathcal{O}(\mathbb{C}^\times))\!: \si(b)=b\}.
$$
Identifying $\mathrm{M}_2(\mathcal{O}(\mathbb{C}^\times))$ with $\mathcal{O}(\mathbb{C}^\times,\mathrm{M}_2)$ we note that all matrix-valued  functions contained in $\widehat{\mathbb{C}\Ga}$ commute with~$u$ (equivalently, are  diagonal  with respect to the basis of eigenvectors of~$u$) at the points~$1$ and~$-1$.

Note that $\mathbb{C}\Ga$ can be described as the universal algebra generated by two idempotents, which can be identified with $(1-u)/2$ and $(1-v)/2$. Banach algebras generated by two idempotents are investigated in the theory of Banach PI-algebras, see, e.g., \cite{RS11}.
\end{exm}

The second example is a modification of the first.
\begin{exm}\label{Z2Zsd}
Consider the semi-direct product $\mathbb{Z}_2\ltimes\mathbb{C}$ of Lie groups, where~$\mathbb{Z}_2$ (with generator~$u$) acts on~$\mathbb{C}$ by $u\cdot z=-z$. It is not hard to see that $\widehat{\mathscr A}(\mathbb{Z}_2\ltimes\mathbb{C})$ is isomorphic to a subalgebra of
$\mathcal{O}(\mathbb{C}, \mathrm{M}_2)$.
The details are left to the reader.
\end{exm}

\subsection*{Simply-connected  nilpotent  groups}

In this section we give an explicit description of $\widehat{\mathscr A}(G)$ in the case when~$G$ is simply connected  and nilpotent and also consider an example.
The case of an arbitrary connected complex Lie group is discussed in the next section.

Let $G$ be a simply-connected  nilpotent complex Lie group and~$\mathfrak{g}$ its Lie algebra.
It follows from Proposition~\ref{ANenvaff} that $\widehat U(\mathfrak{g})$ is an HFG  Hopf algebra. Since   by Proposition~\ref{csicHFG} $\widehat{\mathscr A}(G)\cong \widehat U(\mathfrak{g})$, it suffices to describe $\widehat U(\mathfrak{g})$. Such a description is given in the author's paper \cite[Theorem~1.1]{ArAMN}. Here we modify it and write elements of $\widehat U(\mathfrak{g})$ in a slightly different but more concrete form.

To do this consider the decreasing filtration  $\mathscr{F}$ on~$\mathfrak{g}$ defined  by $\mathfrak{g}_1\!:=\mathfrak{g}$ and $\mathfrak{g}_i\!:=[\mathfrak{g}, \mathfrak{g}_{i-1}]$, i.e.,  the lower central series of~$\mathfrak{g}$.
For any linear
basis $e_1,\ldots, e_m$ of~$\mathfrak{g}$ we put
\begin{equation}\label{widef}
w_i\!:=\max\{j:\,e_i\in \mathfrak{g}_j \}\quad\text{and}\quad
w(\alpha)\!:=\sum_i w_i\alpha_i \,,
\end{equation}
where $\alpha=(\alpha_1,\ldots,\alpha_m)\in\mathbb{Z}_+^m$. If $w_i\le w_{i+1}$  for all $i$
and $\mathfrak{g}_j = \mathrm{span}\{e_i : w_i \ge j\}$ for all $j$, then
$(e_i)$ is called an
\emph{$\mathscr{F}$-basis}.  (This terminology is used in
\cite{Go78,Go79,Pir_stbflat,ArAMN}.) Consider the PBW-basis $(e^\al\,:\al\in\mathbb{Z}_+^m)$  in
$U(\mathfrak{g})$ associated with an $\mathscr{F}$-basis $(e_i)$ and denote by $[U(\mathfrak{g})]$  the set of formal
series with respect to $(e^\al)$.

\begin{thm}\label{nilpenv}
If $G$ is a simply-connected  nilpotent complex Lie group and $e_1,\ldots, e_m$ is  an
$\mathscr{F}$-basis in its Lie algebra~$\mathfrak{g}$, then the natural Hopf $\mathbin{\widehat{\otimes}}$-algebra homomorphism $U(\mathfrak{g})\to {\mathscr A}(G)$ is extended to a topological isomorphism of locally convex spaces
\begin{equation}\label{Ufgclear}
 \Bigl\{ a=\sum_{\alpha\in\mathbb{Z}_+^m} c_\alpha
e^\alpha\in [U(\mathfrak{g})]\! : \| a\|_r\!:=\sum_\alpha |c_\alpha|
\frac{r^{\al_1+\cdots+\al_m}}{(\alpha_1)!^{w_1-1}\cdots
(\alpha_m)!^{w_m-1}}<\infty \;\forall r>0\Bigr\}\,
\end{equation}
and $\widehat{\mathscr A}(G)$.
\end{thm}

Note that when, in addition, $G$ is abelian, we have $w_1=\cdots=w_m=1$ and then $\widehat{\mathscr A}(G)$
is just the algebra of all entire functions in $m$~variables.

In the argument, we use the space of exponential
analytic functionals introduced in \cite{Ak08}. We first recall what is the space of holomorphic functions of
exponential type. A \emph{submultiplicative weight} on a~locally compact group~$G$ is a locally
bounded function $\om$ on~$G$ such that $\om(gh)\le \om(g)\,\om(h)$ for $g,h\in G$.
A~holomorphic function~$f$ on a complex Lie group~$G$ is said to
be of \emph{exponential type} if there is a submultiplicative
weight~$\om$ satisfying  $|f(g)|\le \om(g)$ for all $g\in G $. The
linear space  of all holomorphic functions  of
exponential type on~$G$ is denoted by $\mathcal{O}_{exp}(G)$
\cite[Section~5.3.1]{Ak08}.

In what follows we consider only  compactly generated complex Lie groups. In this case it is natural to endow
$\mathcal{O}_{exp}(G)$ with an inductive topology. To be precise, for any
locally bounded function
$\up\!:G\to [1,+\infty)$ put
\begin{equation*}
\mathcal{O}_\up(G)\!:=\Bigl\{ f\in\mathcal{O}(G) \!: |f|_\up\!:=\sup_{g\in
G}{\up(g)}^{-1}{|f(g)|}<\infty\Bigr\}.
\end{equation*}
Then we consider $\mathcal{O}_{exp}(G)$ as the inductive limit of the system of Banach spaces $\mathcal{O}_\om(G)$, where $\om$ runs over submultiplicative weights.

Following \cite{Ak08} we denote the strong dual space of $\mathcal{O}_{exp}(G)$
by ${\mathscr A}_{exp}(G)$ and call its elements \emph{exponential
analytic functionals}. Both $\mathcal{O}_{exp}(G)$ and ${\mathscr A}_{exp}(G)$ are Hopf
$\mathbin{\widehat{\otimes}}$-algebras; see \cite[Theorem 5.12]{Ak08} or \cite[Propositions 3.1 and Proposition 3.3]{Ar20+B}. The natural embedding
$\mathcal{O}_{exp}(G)\to \mathcal{O}(G)$ is a Hopf
$\mathbin{\widehat{\otimes}}$-algebra homomorphism and so is the dual map ${\mathscr A}(G)\to {\mathscr A}_{exp}(G)$.

We need the following result of Akbarov.
\begin{thm}\label{AMEAG}
\cite[Theorem 6.2]{Ak08} If $G$ is a compactly generated complex Lie group,
then the natural homomorphism ${\mathscr A}(G)\to {\mathscr A}_{exp}(G)$ is an Arens--Michael envelope.
\end{thm}

We also need the predual Banach space of $\mathcal{O}_\up(G)$ defined above.
Let $\up\!:G\to [1,+\infty)$ be a locally bounded function. Following \cite{Ak08} (see also \cite{ArAMN}) we denote by~$V_\up$  the closure of the absolutely convex
hull of
$$
\{\up(g)^{-1}\de_g:\,g\in G\}
$$
in ${\mathscr A}(G)$.  Let $\|\cdot\|_{\up}$  be the Minkowski functional of
$V_\up$.
We denote by  ${{\mathscr A}}_\up(G)$  the completion of ${{\mathscr A}}(G)$ with respect to
$\|\cdot\|_{\up}$ and by  ${{\mathscr A}}_{\up^\infty}(G)$  the completion of ${{\mathscr A}}(G)$ with respect to the
sequence  $(\|\cdot\|_{\up^n};\, n\in\mathbb{N})$ of seminorms, where $\up^n(g):= \up(g)^n$.

In the following lemma we use
the power series algebra  $\mathfrak{A}_s$ ($s\ge 0$) defined in~\eqref{faAsdef}.

\begin{lm}\label{1dimdsp}
 If $\up(z)\!:=\exp(|z|^{1/p})$ for some $p\ge 1$, then
${\mathscr A}_{\up^\infty}(\mathbb{C})$ is topologically isomorphic to the power
series algebra $\mathfrak{A}_{p-1}$.
\end{lm}
\begin{proof}
We put $\mathcal{O}_{\up^\infty}(\mathbb{C})\!:=\bigcup_{n\in\mathbb{N}}
\mathcal{O}_{\up^n}(\mathbb{C})$ and consider this space with the inductive limit topology. It is easy to see that the space ${\mathscr A}_{\up^\infty}(\mathbb{C})$  is nuclear and hence reflexive.  
Then by \cite[Lemma~2.11]{ArAMN},  the pairing between $\mathcal{O}(\mathbb{C})$ and ${\mathscr A}(\mathbb{C})$ induces the pairing between  $\mathcal{O}_{\up^\infty}(\mathbb{C})$ and ${\mathscr A}_{\up^\infty}(\mathbb{C})$ and, moreover, ${\mathscr A}_{\up^\infty}(\mathbb{C})$ is the strong dual space of $\mathcal{O}_{\up^\infty}(\mathbb{C})$.
It follows from the classical theory of order and type for entire function that  this strong dual space is topologically isomorphic~to
$$
\Bigl\{a=\sum_{n=0}^\infty  a_n x^n\! :\,
\lim_{n\to \infty} |a_n|\, n!\, n^{-np}\,r^n=0
\;\forall r>0\Bigr\}
$$
with the topology determined by
$|a|_{r,p}\!:=\sup_{n\in\mathbb{Z}_+} |a_n| \,n! \,n^{-np}\,r^n$. The isomorphism is given by the pairing $\langle x^n, f\rangle\!:=d^nf/dz^n(0)$;
see, for example, \cite[Theorem~1.1.11]{Le81}.
Stirling's formula implies that replacing the factor $ n!\, n^{-np} $
with $n!^{1-p}$ in the definition of these seminorms, we get the same topology.

Put $b_{nk}\!:=n!^{1-p}k^n$. Then $(b_{nk})_{n,k\in\mathbb{N}}$ is a
K\"{o}the matrix and for each $k\in\mathbb{N}$ there is $n\in\mathbb{N}$ such that $\sum_{j}b_{jk}b_{jn}^{-1}<\infty$. Therefore, replacing  the supremum with the sum in the already modified definition of the seminorms, we again get the same topology
\cite[Propositions~27.15 and~27.16]{MeVo}. By~\eqref{faAsdef}, this implies that ${\mathscr A}_{\up^\infty}(\mathbb{C})\cong\mathfrak{A}_{p-1}$.
\end{proof}

\begin{proof}[Proof of Theorem~\ref{nilpenv}]
By Akbarov's result (Theorem~\ref{AMEAG}), it suffices to show that the power series space in~\eqref{Ufgclear} is isomorphic to ${\mathscr A}_{exp}(G)$.

Let $\ell$ be a word length function on~$G$, i.e.,
$\ell(g)\!: = \min \{ n \!: \, g \in U^{n} \}$, where $U$ is a relatively compact generating set.
Put $\xi(g)\!:=\exp \ell(g)$. Then we have an
isomorphism ${\mathscr A}_{exp}(G)\cong{\mathscr A}_{\xi^{\infty}}(G)$; see \cite[Proposition~2.8]{ArAMN}. (In fact, this assertion is a reformulation of \cite[Theorem 5.3]{Ak08}.)

Denote by $t_1,\ldots,t_m$ the canonical coordinates of the second kind on~$G$ with respect to the basis $e_1,\ldots,e_m$ and
set
$$
\up(g)\!:=\exp(|t_1|^{1/w_1}+\cdots+|t_m|^{1/w_m}).
$$
Note that the function $g\mapsto|t_1|^{1/w_1}+\cdots+|t_m|^{1/w_m}$  is equivalent to
$g\mapsto\max\{|t_1|^{1/w_1},\ldots,|t_m|^{1/w_m}\}$, which, in turn, is equivalent
to $\ell$ by~\cite[Theorem~3.1]{ArAMN}; see a proof and remarks on the history of the last result in  [ibid.], Appendix and Introduction, resp. The equivalence implies  that ${\mathscr A}_{\xi^{\infty}}(G)\cong{\mathscr A}_{\up^{\infty}}(G)$ [ibid., Lemma~3.3].

Denote by $G_j$ the subgroup $\{\exp (t_je_j)\!:\,t_j\in\mathbb{C}\}$ and
put $\up_j(t_j)\!:=\exp(|t_j|^{1/\om_j})$. Since $\up(g)=\up_1(t_1)\cdots \up_m(t_m)$, we have
by \cite[Proposition 5.5]{Ar19} that
$$
{\mathscr A}_{\up^{\infty}}(G)\cong {\mathscr A}_{\up_1^{\infty}}(G_1)\mathbin{\widehat{\otimes}} \cdots \mathbin{\widehat{\otimes}} {\mathscr A}_{\up_m^{\infty}}(G_m)
$$
as a locally convex space. Since each $G_j$ is isomorphic to~$\mathbb{C}$, it follows from Lemma~\ref{1dimdsp} that ${\mathscr A}_{\up_j^{\infty}}(G_j)\cong \mathfrak{A}_{w_j-1}$, where $\mathfrak{A}_s$ is defined in \eqref{faAsdef}. Since each $\mathfrak{A}_{w_j-1}$ is a K\"{o}the space, we have from  \cite[Proposition~3.3]{Pi02}
that $ \mathfrak{A}_{w_1-1}\mathbin{\widehat{\otimes}} \cdots \mathbin{\widehat{\otimes}} \mathfrak{A}_{w_m-1}$ is also a K\"{o}the space and topologically isomorphic to the space described in \eqref{Ufgclear}.
\end{proof}

\begin{rem}\label{wawa}
There is an alternative description mentioned above:
\begin{multline*}
\widehat{\mathscr A}(G)\cong {\mathscr A}_{exp}(G)\cong
\widehat U(\mathfrak{g})\\
\cong \Bigl\{ x=\sum_{\alpha\in\mathbb{Z}_+^m} c_\alpha
e^\alpha\in [U(\mathfrak{g})]\! : \| x\|'_r\!:=\sum_\alpha |c_\alpha| \alpha!\,
w(\al)^{-w(\al)} r^{w(\alpha)}<\infty \;\forall r>0\Bigr\},
\end{multline*}
where $w(\al)$ is defined in \eqref{widef}; see
\cite[Theorem~1.1]{ArAMN}.
Note that the seminorms in neither this family nor in~\eqref{Ufgclear} are necessarily submultiplicative; but there exists a family of submultiplicative seminorms that is equivalent to each of them.
\end{rem}

\begin{exm}
Let $\mathfrak{n}_{3,3}$ be a free $3$-step nilpotent Lie algebra in $3$
generators. Then we have a vector space decomposition
$$
\mathfrak{n}_{3,3}=\mathfrak{l}\oplus[\mathfrak{l},\mathfrak{l}]\oplus[\mathfrak{l},[\mathfrak{l},\mathfrak{l}]]\,,
$$
where $\mathfrak{l}$ is the generating subspace. Fix a basis $\{e_1,e_2,e_3\}$ in~$\mathfrak{l}$. Then
$$
e_4\!:=[e_1,e_2],\qquad e_5\!:=[e_1,e_3],\qquad e_6\!:=[e_2,e_3]
$$
form a basis in $[\mathfrak{l},\mathfrak{l}]$. Also,
$$
\begin{matrix}
[e_1,e_4],&[e_1,e_5],&[e_1,e_6],\\
[e_2,e_4],&[e_2,e_5],&[e_2,e_6],\\
[e_3,e_4],&[e_1,e_5]\,&
\end{matrix}
$$
form a basis in $[\mathfrak{l},[\mathfrak{l},\mathfrak{l}]]$, which we denote by $e_7,\ldots,e_{14}$.

Then by Proposition~\ref{Ufgclear}, the  HFG  Hopf algebra $\widehat{U}(\mathfrak{n}_{3,3})$ is topologically isomorphic to
\begin{equation*}
 \Bigl\{ a=\sum_{\alpha\in\mathbb{Z}_+^{14}} c_\alpha
e^\alpha\in [U(\mathfrak{n}_{3,3})]\! :\,
\sum_\alpha |c_\alpha|\,
\frac{r^{\al_1+\cdots+\al_{14}}}{\alpha_4!\,\alpha_5!\,\alpha_6!\,(\alpha_7!\cdots\alpha_{14}!)^2}<\infty \;\forall r>0\Bigr\}.
\end{equation*}
\end{exm}

\subsection*{General connected groups}

Now we write an explicit form of ${\mathscr A}_{exp}(G)$ for an arbitrary connected complex Lie group $G$.
Consider first the quotient homomorphism $\si\!:G\to
G/\LinC(G)$, where $\LinC(G)$ is the intersection of the
kernels of all finite-dimensional holomorphic representations of $G$ (the linearizer), and note that
the induced Hopf $\mathbin{\widehat{\otimes}}$-algebra homomorphism ${\mathscr A}_{exp}(G)\to {\mathscr A}_{exp}(G/\LinC(G))$ is, in fact,
a $\mathbin{\widehat{\otimes}}$-algebra  isomorphism \cite[Theorem~5.3(A)]{Ar19} and hence a Hopf $\mathbin{\widehat{\otimes}}$-algebra isomorphism. Since here we are interested in ${\mathscr A}_{exp}(G)$ only, we can assume that $\LinC(G)$ is trivial, i.e., $G$ is linear.

We describe  ${\mathscr A}_{exp}(G)$ in a form that generalizes \eqref{Ufgclear}. To do this we need some structure theory of linear complex Lie groups. Suppose that~$G$ is connected and linear. The \emph{exponential radical} of~$G$ is the normal complex Lie subgroup $E$  such that $G/E$ is the largest quotient of $G$ that is locally isomorphic to a direct product of a nilpotent and semisimple complex Lie group. (This is a partial case of a more general definition in the case of real Lie group can be found in~\cite{Co08}; see a discussion in~\cite[Section~3]{Ar19}.)

Let $\mathfrak{g}$  and $\mathfrak{e}$  be the Lie algebras of~$G$  and $E$. Then $\mathfrak{e}$ is uniquely defined by the condition that $\mathfrak{g}/\mathfrak{e}$ is the largest quotient of $\mathfrak{g}$ that
is a direct sum of a semisimple algebra and a nilpotent algebra; see \cite[Lemma~3.5, Lemma~3.7 and Corollary~3.10]{Ar19}.
It is not hard to see that~$E$ is nilpotent, simply connected and integral \cite[Proposition 1.3]{Ar23B}.

\begin{rem}
In our case when $G$ is connected and linear, the ideal $\mathfrak{e}$ has the following explicit description:
\begin{equation}\label{rinsr}
\mathfrak{e}=\mathfrak{r}_\infty+(\mathfrak{s},\mathfrak{r}),
\end{equation}
where  $\mathfrak{r}$ is the solvable radical, $\mathfrak{r}_\infty$ is the intersection of the lower central series of $\mathfrak{r}$,
$\mathfrak{s}$ a Levi complement, and
$(\mathfrak{s},\mathfrak{r})$ is the Lie subalgebra (in fact, a Lie ideal) generated by the linear subspace $[\mathfrak{s},\mathfrak{r}]$ (the linear span of elements of the form $[X,Y]$, where $X\in \mathfrak{s}$ and $Y\in \mathfrak{r}$); see \cite [Corollary 3.10]{Ar19}.
Note that in my article \cite{Ar19} the notation $[\cdot,\cdot]$ was used instead of $(\cdot,\cdot)$. It seems that it was not a good choice because $[\cdot,\cdot]$ is traditionally taken for linear spans.
\end{rem}

First, we  write the underlying space of ${\mathscr A}_{exp}(G)$  as a tensor product. Indeed,
since $G$ is connected and linear,
$G=B\rtimes L$, where~$B$  is simply connected solvable and~$L$ is
linearly complex reductive \cite[Theorem~16.3.7]{HiNe}. Furthermore, $E\subset B$ \cite[Theorem 3.14]{Ar19} and
there is a complex manifold isomorphism
$$
L\times (B/E)\times E\to G
$$
that induces
a $\mathbin{\widehat{\otimes}}$-algebra isomorphism
\begin{equation}\label{oexpGdec}
{\mathcal R}(L) \mathbin{\widehat{\otimes}} \mathcal{O}_{exp}(B/E)\mathbin{\widehat{\otimes}} {\mathcal R}(E)\cong \mathcal{O}_{exp}(G),
\end{equation}
where as above ${\mathcal R}(\cdot)$ denotes the algebra  of regular functions;
see \cite[Theorems~4.1 and ~5.10]{Ar19}.
Applying the strong duality functor we get locally convex space isomorphisms
\begin{multline}\label{AexpGdec}
{\mathscr A}_{exp}(G) \cong({\mathcal R}(L) \mathbin{\widehat{\otimes}} \mathcal{O}_{exp}(B/E)\mathbin{\widehat{\otimes}} {\mathcal R}(E))'\\
\cong {\mathcal R}(L)' \mathbin{\widehat{\otimes}} {\mathscr A}_{exp}(B/E)\mathbin{\widehat{\otimes}} {\mathcal R}(E)',
\end{multline}
where the second isomorphism is well defined because all the tensor multiplies in~\eqref{oexpGdec} are complete nuclear (DF)-spaces.

Further, we  take a close look at each multiple in the right-hand side of~\eqref{AexpGdec} separately.

Since $L$ is
linearly complex reductive,
there is a compact real Lie subgroup~$K$ such that~$L$ is the universal
complexification of~$K$. Recall that algebraic (=rational) irreducible representations of~$L$ are in one-to-one correspondence with unitary irreducible representations of~$K$. Denote by~$\Si$ the (countable) set of all non-isomorphic unitary irreducible representations of~$K$. Each~$\si\in\Si$ is finite dimensional and unitary (when restricted to~$K$). We  fix a finite orthonormal  basis  in the representation space of~$\si$ and  denote by $u_{ij}^\si$ the corresponding matrix coefficients multiplied by the dimension~$d_\si$ of~$\si$.

The fact that $L$ is linearly complex reductive implies that ${\mathcal R}(L)=\bigoplus_{\si\in\Si} T_\si$, where $T_\si$ is the linear space of matrix coefficients of~$\si$ \cite[Section 3.1, p.~180, Theorem]{Pr07}. Consider matrix coefficients as functions on~$K$  and also the convolution~$\ast$ with respect to the normalized left Haar measure on~$K$.
It follows from  standard orthogonality relations \cite[Theorem~(27.17)]{HeRo2}, that $u_{ij}^\si\ast u_{jk}^\si=u_{ik}^\si$ and $u_{ij}^\si\ast u_{j'k}^{\si'}=0$ when $j\ne j'$ or $\si\ne\si'$.
So for any $\si$ the subset $\{u_{ij}^\si\}$ generates the convolution algebra isomorphic to the matrix algebra $M_{d_\si}$, which in turn can be naturally identified with~$T_\si'$. Thus, ${\mathcal R}(L)'\cong \prod_{\si\in\Si} M_{d_\si}$ as a locally convex vector space and so  the first multiple in~\eqref{AexpGdec} is described.

Let $\mathfrak{b}$ be the Lie algebra of~$B$. Note that $\mathfrak{e}$ is an ideal in $\mathfrak{b}$ and both $\mathfrak{b}/\mathfrak{e}$ and $\mathfrak{e}$ are nilpotent \cite[Lemma 3.4]{Ar19} and so are $B/E$ and $E$, which are also simply connected. Hence the structure of the second multiple ${\mathscr A}_{exp}(B/E)$ in~\eqref{AexpGdec} is given by Theorem~\ref{nilpenv}. On the other hand, ${\mathcal R}(E)$ can be identified with $\mathbb{C}[f_1,\ldotp,f_k]$, where  $f_1,\ldots,f_k$ is a basis of~$\mathfrak{e}$. Then the third multiple $ {\mathcal R}(E)'$  in~\eqref{AexpGdec} is isomorphic to
$[U(\mathfrak{e})]$, the linear space of all formal  series in $f^\be$, where $\be\in \mathbb{Z}_+^k$.

Finally, we have that
\begin{equation}\label{Aexp3mu}
{\mathscr A}_{exp}(G) \cong\left(\prod_{\si\in\Si} M_{d_\si}\right) \mathbin{\widehat{\otimes}}  U_{\mathscr M}(\mathfrak{b}/\mathfrak{e}) \mathbin{\widehat{\otimes}} [U(\mathfrak{e})],
\end{equation}
 where $U_{\mathscr M}$ denote the space in~\eqref{Ufgclear}, as a locally convex space.

Fix now  an $\mathscr{F}$-basis $e_1,\ldots, e_m$ of~$\mathfrak{b}/\mathfrak{e}$.
Note that $u_{ij}^\si$,  $e^\alpha$  and $f^\beta$, where $\al$ and $\be$ are multi-indices, are continuous linear functionals on ${\mathcal R}(L)$, $\mathcal{O}_{exp}(B/E)$ and ${\mathcal R}(E)$, respectively. Therefore, for any tuple of indices,
$u_{ij}^\si \otimes e^\alpha\otimes f^\beta$ can be identified with a  continuous linear functional on $\mathcal{O}_{exp}(G)$. We denote it by $u_{ij}^\si e^\alpha f^\beta$.

It is easy to see that all three multiples in~\eqref{Aexp3mu} are K\"{o}the spaces.
So applying again a standard argument from the K\"{o}the space theory \cite[Proposition~3.3]{Pi02}  and using Akbarov's isomorphism  $\widehat{\mathscr A}(G)\cong {\mathscr A}_{exp}(G)$ (Theorem~\ref{AMEAG}),
we obtain the following result.

\begin{thm}\label{genpenv}
Let $G$ be a connected linear complex Lie group. Then the set of functionals of the form $u_{ij}^\si e^\alpha f^\beta$ is a topological basis of~${\mathscr A}_{exp}(G)$ and \emph{(}in the above notation and the notation in~\eqref{widef}\emph{)}
\begin{multline*}
\widehat{\mathscr A}(G)\cong {\mathscr A}_{exp}(G)= \Bigl\{ \sum c_{ij}^{\si\alpha\be} u_{ij}^\si
 e^\alpha f^\beta\! :\\
 \| a\|_{r,F}\!:=\sum_{\alpha\in \mathbb{Z}_+^m}\, \sum_{(\be,\si) \in F} \,
\bigl|c_{ij}^{\si\alpha\be}\bigr|\,
\frac{r^{\al_1+\cdots+\al_m}}{(\alpha_1)!^{w_1-1}\cdots
(\alpha_m)!^{w_m-1}}<\infty \;\forall r>0,\,\forall F\Bigr\},
\end{multline*}
where $F$ runs over finite subsets of $ \mathbb{Z}_+^k \times\Si$.
\end{thm}

Note that is possible to use another family of seminorms as in Remark~\ref{wawa}.

\begin{rem}
(A)~Theorem~\ref{genpenv} gives a description of $\widehat{\mathscr A}(G)$  as a locally convex space. The Hopf $\mathbin{\widehat{\otimes}}$-algebra structure on $\widehat{\mathscr A}(G)$ is inherited from that of ${\mathscr A}(G)$. In a more concrete way it can be characterized by considering $\widehat{\mathscr A}(G)$ as an analytic smash product (see a definition in \cite{Pi4}). The multiples in this product are $\widehat{\mathscr A}(L)$, which can be identified with $\prod_{\si\in\Si} M_{d_\si}$, and a completion of $U(\mathfrak{b})$. For details see~\cite{ArPC15}.

(B)~It is hot hard to see that if $w_j>1$ for some $j$, then  $\widehat{\mathscr A}(G)$ contains a copy of $\mathfrak{A}_{w_j-1}$ as a Hopf $\mathbin{\widehat{\otimes}}$-subalgebra; see \cite[Theorem 6.3 and Remark 6.8]{ArPC15} for a general result and discussion.  But in this case $\mathfrak{A}_{w_j-1}$ is not HFG  by \cite[Proposition~10]{ArRC}; cf. Example~\ref{commnHFG}.
\end{rem}

\begin{exm}\label{GLsm}
Consider  a $6$-dimensional vector space $\mathfrak{n}$ with basis\\ $e_1,f_1,f_2,f_3,f_4,f_5$
and commutation relations
$$
[e_1, f_1] = f_2,\quad [e_1 , f_2] = f_5, \quad [f_3 , f_4] = f_5
$$
with the undefined brackets being zero.  Then $\mathfrak{n}$ is a complex Lie algebra because it can be written as a semidirect sum $\mathfrak{h}_1\ltimes \mathfrak{h}_2$, where $\mathfrak{h}_1 = \mathrm{span}\{e_1\}$ and $\mathfrak{h}_2 = \mathrm{span}\{f_1,f_2,f_3,f_4,f_5\}$.

The  standard  action of $\mathrm{GL}_2(\mathbb{C})$
on the linear span of $\{f_3,f_4\}$  can be uniquely extended to the action on~$\mathfrak{n}$ that is trivial on~$e_1$ and coincides with the multiplication on the determinant on~$f_1$,~$f_2$ and~$f_5$ (see a description of the automorphism group in \cite[p.~36]{Go98}, where~$\mathfrak{n}$ is denoted by~$N_{6,3,2}$). Then we can consider the corresponding action of $\mathrm{GL}_2(\mathbb{C})$ on~$N$, the simply-connected Lie group associated with~$\mathfrak{n}$.

Let $G$ be the semidirect product $\mathrm{GL}_2(\mathbb{C})\ltimes N$. It is evident that
the Lie algebra of $G$ is isomorphic to
$\mathfrak{g}\mathfrak{l}_2(\mathbb{C})\ltimes \mathfrak{n}$. Denote its solvable radical by $\mathfrak{r}$. Since the center of $\mathrm{GL}_2(\mathbb{C})$ acts on $f_1$, $f_2$ and $f_5$ with a non-trivial character, so does the center of its Lie algebra $\mathfrak{g}\mathfrak{l}_2(\mathbb{C})$. Therefore  $f_1$, $f_2$ and $f_5$ lie in each term of the lower central series of $\mathfrak{r}$ and hence $\mathfrak{r}_\infty=\mathrm{span}\{f_1,f_2,f_5\}$.  Also, note that $\mathfrak{s}\mathfrak{l}_2(\mathbb{C})$ is a Levi subalgebra of $\mathfrak{g}\mathfrak{l}_2(\mathbb{C})\ltimes \mathfrak{n}$. Since $(\mathfrak{s}\mathfrak{l}_2(\mathbb{C}),\mathfrak{r})=\mathrm{span}\{f_3,f_4,f_5\}$, it follows from~\eqref{rinsr} that $\mathfrak{e}=\mathrm{span}\{f_1,f_2,f_3,f_4,f_5\}$.
Further, $\mathfrak{b}=\mathfrak{n}$ and we can
use the same notation for the images of $e_1,f_1,f_2$ in $\mathfrak{b}/\mathfrak{e}$.

Denote by~$\Si$ the set of equivalences classes of algebraic irreducible representations of $\mathrm{GL}_2(\mathbb{C})$ (which can be identified with unitary  irreducible representations of $\mathrm{U}(2)$). Then
we have from Theorem~\ref{genpenv} that the  HFG Hopf algebra $\widehat{\mathscr A}(G)$ is isomorphic to
$$
 \left\{\,\sum c_{ij}^{\si\alpha\be} u_{ij}^\si
 e^\alpha f^\beta\! :\,
 \sum_{\alpha\in\mathbb{Z}_+}\, \sum_{(\be,\si)\in F}
\bigl|c_{ij}^{\si\alpha\be}\bigr|\,
r^{\al_1}<\infty \;\;\forall r>0\,\forall F\right\},
$$
where $F$ runs over finite subsets of $ \mathbb{Z}_+^5 \times\Si$, as a locally convex space.
 \end{exm}

\section{Examples of quantum complex Lie groups}
\label{eqclg}

A quantum group having an affine Hopf algebra over $\mathbb{C}$ as the space of `regular functions'
gives an example of an  HFG Hopf algebra when passing to the Arens--Michael envelope (see Proposition~\ref{ANenvaff}).
But in this section, we concentrate on the holomorphic analogues of Hopf algebras corresponding to some $\hbar$-adic quantum groups. We consider the deformation parameter $\hbar$ as a complex number not a formal variable. (We speak on deformation in a broad sense meaning that not only the multiplication but also the underlying topological
vector spaces can be deformed.)

Direct checking of the Hopf algebra axioms is sometimes cumbersome and so is for topological Hopf algebras.
Some simplification is possible with using topological generators.
\begin{lm}\label{chegen}
Let $H$ be a unital $\mathbin{\widehat{\otimes}}$-algebra that is topologically generated by a subset~$X$. Suppose that
$\De\!:H\to H\mathbin{\widehat{\otimes}} H$ and $\varepsilon\!:H\to \mathbb{C}$  are continuous unital homomorphisms.

\emph{(A)}~If $\De$ is coassociative on~$X$ and the counit axiom for $\varepsilon$ holds  on~$X$, then $H$ is a $\mathbin{\widehat{\otimes}}$-bialgebra.

\emph{(B)}~If, in addition, $S\!:H\to H$ a continuous antihomomorphism such that the antipode axiom holds on~$X$, then $H$ is a Hopf $\mathbin{\widehat{\otimes}}$-algebra.
\end{lm}
\begin{proof}
It is not hard to check that $H$ is a $\mathbin{\widehat{\otimes}}$-coalgebra; this can be done as in the algebraic case, cf. \cite[Lemma 5.1.1]{Ra12}. Since $\De$ and $\varepsilon$ are homomorphisms,  $H$ is a $\mathbin{\widehat{\otimes}}$-bialgebra. The rest is clear.
\end{proof}

\subsection*{A deformation of universal enveloping algebra: the semisimple case}

In this section we define an HFG Hopf algebra, which can be considered as  an analytic version of the $\hbar$-adic quantized algebra $U_\hbar(\mathfrak{s}\mathfrak{l}_2)$. Let the quantization parameter $\hbar$ belongs to $\mathbb{C}$ and $\sinh\hbar\ne 0$.
We consider the universal HFG algebra generated by $E$, $F$ and $H$  subject to relations
\begin{equation}\label{EFHqun}
[H,E]=2E, \quad [H,F]=-2F, \quad [E,F]=\frac{\sinh \hbar H}{\sinh
\hbar}.
\end{equation}
(It is isomorphic to the quotient of the algebra of free entire functions in generators $E$, $F$ and $H$ over the closed two-sided ideal generated by the corresponding identities.)
We denote this algebra by $\widetilde U(\mathfrak{s}\mathfrak{l}_2)_\hbar$.

\begin{pr}\label{sl2qco}
Let  $\hbar\in\mathbb{C}$ and $\sinh\hbar\ne 0$. Then $\widetilde
U(\mathfrak{s}\mathfrak{l}_2)_\hbar$ is an infinite-dimensional   HFG Hopf algebra
with respect to the comultiplication $\De$, counit $\varepsilon$ and
antipode $S$ determined by
\begin{alignat*}{3}
\De\!&:H\mapsto H\otimes 1 + 1\otimes H,& \quad E &\mapsto E \otimes e^{\hbar H} + 1 \otimes E,& \quad F&\mapsto F \otimes 1 + e^{-\hbar H} \otimes F\,;\\
S\!&:H\mapsto -H,& E&\mapsto -E e^{-\hbar H} , & F &\mapsto -e^{\hbar H}F\,;\\
 \varepsilon\!&:H,E,F\mapsto 0\,.&&&&
\end{alignat*}
Moreover, $S$ is invertible.
\end{pr}
The relations in~\eqref{EFHqun} and formulas for operations are the same as in the $\hbar$-adic case \cite[\S~3.1.5]{KSc}.

To prove that $\De$,  $\varepsilon$ and $S$ are well defined we use Lemma~\ref{chegen} and the following lemmas. The first can  easily be proved by using Taylor's expansion.
\begin{lm}\label{expprim}
Let $A_1$ and $A_2$ be Arens--Michael algebras, $a_1\in A_1$, $a_2\in A_2$ and $a\!=a_1\otimes 1+1\otimes a_2$ be the element of $ A_1\mathbin{\widehat{\otimes}} A_2$.  Then  $e^a=e^{a_1}\otimes e^{a_2}$.
 \end{lm}

The second lemma comes from the noncommutative analysis. The proof is included for completeness.
\begin{lm}\label{draddcotomu}
Suppose that $a$ and $b$ are elements of an Arens--Michael algebra such that
$[b,c]=\al c$ for some $\al\in\mathbb{C}$. Then
$e^{\la b}\,c\, e^{-\la b}=e^{\la\al}c$ for any $\la\in\mathbb{C}$.
\end{lm}
\begin{proof}
We recall first the standard formula
\begin{equation}\label{adhQ}
\mathop{\mathrm{ad}}\nolimits h(b)(c) = \sum_{n=1}^\infty \frac{1}{n!}  (\mathop{\mathrm{ad}}\nolimits b)^n(c)h^{(n)}(b),
\end{equation}
which holds for arbitrary elements $b$ and $c$ of an Arens--Michael
algebra and every entire function~$h$; see e.g., \cite[\S~15,
p.~82, Corollary~1]{BS01}, where the formula is proved for bounded
operators in Banach spaces. (Here $\mathop{\mathrm{ad}}\nolimits b(c)\!:=[b,c]$.)

Suppose now that $[b,c]=\al c$. Then
$(\mathop{\mathrm{ad}}\nolimits b)^n(c)=\al^n c$. Hence,
$$
\mathop{\mathrm{ad}}\nolimits h(b)c = \sum_{n=1}^\infty \frac{1}{n!}  \al^n c\, h^{(n)}(b)=c\,(h(b+\al)-h(b)) \,,
$$
i.e.,
$h(b)\,c=c\,h(b+\al)$.
In particular, if $h(z)\!:=e^{\la z}$, then $e^{\la b}\,c=c\,e^{\la (b+\al)}$, what is to be demonstrated.
\end{proof}

\begin{proof}[Proof of Proposition~\ref{sl2qco}]
Put
$$
H'\!:=H\mapsto H\otimes 1 + 1\otimes H, \quad E'\!:= E \otimes e^{\hbar H} + 1 \otimes E, \quad F'\!:= F \otimes 1 + e^{-\hbar H} \otimes F.
$$
It is easy to see that $[H',E']=2E'$ and $[H',F']=-2F'$. Applying Lemma~\ref{draddcotomu} for $H$ and $E$ with $\al=2$ and $\la=\hbar/2$, we have $E\, e^{-\hbar H/2}=e^{\hbar} e^{-\hbar H/2} E$ and similarly $F\, e^{\hbar H/2}=e^{\hbar} e^{\hbar H/2} F$. These equalities imply
$$[E',F']=\frac{\sinh \hbar H}{\sinh
\hbar}\otimes e^{\hbar H}+ e^{-\hbar H}\otimes\frac{\sinh \hbar H}{\sinh
\hbar}\,.
$$
It follows from Lemma~\ref{expprim} that the expression on the
right is equal to $(\sinh \hbar H')/(\sinh \hbar)$. Since $\widetilde
U(\mathfrak{s}\mathfrak{l}_2)_\hbar$ is a universal algebra, we have a continuous
homomorphism
$$
\De\!: \widetilde U(\mathfrak{s}\mathfrak{l}_2)_\hbar\to \widetilde U(\mathfrak{s}\mathfrak{l}_2)_\hbar\mathbin{\widehat{\otimes}}\widetilde
U(\mathfrak{s}\mathfrak{l}_2)_\hbar
$$
defined by $H\mapsto H'$, $E\mapsto E'$, $F\mapsto F'$.

By Lemma~\ref{expprim}, $\De(e^{\hbar H})=e^{\hbar H}\otimes
e^{\hbar H}$ and this implies that the coassociativity law
for~$\De$ holds on the generators. It is trivial that $\varepsilon$ is a
well-defined continuous homomorphism. To see that the counit axiom
holds on the generators note that $\varepsilon(e^{\hbar H})=\varepsilon(e^{-\hbar
H})=1$. So by Part~(A) of Lemma~\ref{chegen}, $\widetilde
U(\mathfrak{s}\mathfrak{l}_2)_\hbar$ is a $\mathbin{\widehat{\otimes}}$-bialgebra.

In the same manner, we can prove that $S$ is a continuous
antihomomorphism such that the antipode axiom holds on
$\{E,F,G\}$. Thus, by Part~(B) of Lemma~\ref{chegen}, $\widetilde
U(\mathfrak{s}\mathfrak{l}_2)_\hbar$ is a Hopf $\mathbin{\widehat{\otimes}}$-algebra. Similarly, there is a
continuous antihomomorphism defined by
$$
H\mapsto -H,\qquad E\mapsto -e^{-\hbar H}E  , \qquad F \mapsto -Fe^{\hbar H},
$$
that is obviously inverse to~$S$.

It is not hard to see from the representation theory  that $\widetilde
U(\mathfrak{s}\mathfrak{l}_2)_\hbar$ admits infinitely many pairwise
non-equivalent  irreducible finite-dimensional representations
\cite[\S~3.2.3]{KSc}; see also \cite{Ar23}. Thus $\widetilde U(\mathfrak{s}\mathfrak{l}_2)_\hbar$ is
infinite-dimensional.
\end{proof}

The structure of $\widetilde U(\mathfrak{s}\mathfrak{l}_2)_\hbar$ as an Arens--Michael algebra is discussed in~\cite{Ar23}.

\begin{rem}
We can also consider the holomorphic form $\widetilde U(\mathfrak{g})_\hbar$ of
the $\hbar$-adic Drinfeld--Jimbo algebra $U_\hbar(\mathfrak{g})$ for an
arbitrary finite-dimensional complex semisimple Lie algebra~$\mathfrak{g}$
(see, e.g., the definition of the latter in \cite[\S~6.1.3]{KSc}). The set of
generators consists of quantum $\mathfrak{s}\mathfrak{l}_2$-triples (as in
\eqref{EFHqun}) satisfying some additional relations. Following an
argument as above, one can  prove  that $\widetilde U(\mathfrak{g})_\hbar$ is an
infinite-dimensional  HFG Hopf algebra. The details are left to
the reader.
\end{rem}

\subsection*{A deformation of universal enveloping algebra: the non-abelian solvable $2$-dimensional Lie algebra}

In this section we consider a deformation in the case of the non-abelian solvable $2$-dimensional complex Lie algebra. This deformation was proposed in \cite{AiSa} without a discussion on the exact sense of
formulas. Similar relations can be found in \cite[Example 6.2.18]{Ma95}.

Denote by $\mathfrak{a}\mathfrak{f}_1$ the Lie algebra of  the group $\mathrm{Af}_1$
of affine transformations of~$\mathbb{C}$. It is generated by $X$ and $Y$
subject to relation $[X,Y]=Y$. We fix $\hbar\in\mathbb{C}$ such that
$\sinh\hbar\ne 0$ and consider the following deformation of $\widehat
U(\mathfrak{a}\mathfrak{f}_1)$. Denote  by $\widetilde U(\mathfrak{a}\mathfrak{f}_1)_\hbar$ the HFG algebra
generated by~$X$ and~$Y$  subject to relation
$$
[X,Y]=\frac{\sinh \hbar Y}{\sinh
\hbar}\,.
$$

Since the hyperbolic sine has infinitely many zeros, all of order~$1$, it follows from a result of the author
\cite[Theorem~5]{ArRC} that there is a linear (but not algebraic) topological isomorphism
\begin{equation}\label{af1str}
\widetilde U(\mathfrak{a}\mathfrak{f}_1)_\hbar\cong \prod_{n\in \mathbb{Z}} R_n\mathbin{\widehat{\otimes}} \mathcal{O}(\mathbb{C}),
\end{equation}
where every $R_n$ is isomorphic to~$\mathbb{C}[[z]]$.

\begin{pr}\label{af1qco}
Let $\hbar\in\mathbb{C}$ and $\sinh\hbar\ne 0$. Then $\widetilde
U(\mathfrak{a}\mathfrak{f}_1)_\hbar$ is an infinite-dimensional   HFG Hopf algebra
with respect to the comultiplication $\De$,  counit $\varepsilon$ and
antipode $S$ determined by
\begin{alignat*}{2}
\De\!&:X\mapsto X\otimes e^{-\hbar Y}+ e^{\hbar Y}\otimes X,&\qquad &Y\mapsto 1\otimes Y+ Y\otimes 1;\\
 S\!&:X\mapsto -X-\frac{\hbar }{\sinh \hbar}\,\sinh \hbar Y,& & Y\mapsto -Y;\\
 \varepsilon\!&:X,Y\mapsto 0\,.&&
\end{alignat*}
Moreover, $S$ is invertible.
\end{pr}
\begin{proof}
Put
$$
X'\!:=X\otimes e^{-\hbar Y}+ e^{\hbar Y}\otimes X,\qquad
Y'\!:= 1\otimes Y+ Y\otimes 1.
$$
By Lemma~\ref{expprim}, we have $e^{\hbar Y'}=e^{\hbar Y}\otimes
e^{\hbar Y}$. It is easy to check that $[X',Y']=\sinh \hbar
Y/\sinh\hbar$. Since $\widetilde U(\mathfrak{a}\mathfrak{f}_1)_\hbar$ is a universal
algebra, the correspondence $X\mapsto X'$, $Y\mapsto Y'$ induces a
continuous homomorphism
$$
\De\!: \widetilde U(\mathfrak{a}\mathfrak{f}_1)_\hbar\to \widetilde
U(\mathfrak{a}\mathfrak{f}_1)_\hbar\mathbin{\widehat{\otimes}}\widetilde U(\mathfrak{a}\mathfrak{f}_1)_\hbar.
$$

It is easy to see  that the coassociativity law for $\De$ holds on
the generators. As in the proof of Proposition~\ref{sl2qco}, we
see  that $\varepsilon$ is a well-defined continuous homomorphism and the
counit axiom holds on the generators.  Part~(A) of
Lemma~\ref{chegen} implies that $\widetilde U(\mathfrak{a}\mathfrak{f}_1)_\hbar$ is a
$\mathbin{\widehat{\otimes}}$-bialgebra.

It is straightforward that
$$
[S(Y),S(X)]=\frac{\sinh \hbar
S(Y)}{\sinh \hbar}
$$
and so $S$ is a well-defined continuous
anti-automorphism of $\widetilde U(\mathfrak{a}\mathfrak{f}_1)_\hbar$. Note that
$$
m(1\otimes S)(X)=[X,e^{\hbar Y}]-\frac{\hbar}{2\sinh \hbar}( e^{2\hbar Y}-1).
$$
(Here $m$ is the linearization of the multiplication.) Applying~\eqref{adhQ} to $[e^{\hbar Y},X]$, we have only one non-trivial summand on the right-hand side, i.e.,
$$
[e^{\hbar Y},X]=[Y,X]\, \hbar\, e^{\hbar Y}=-\frac{\hbar}{2\sinh \hbar}( e^{2\hbar Y}-1).
$$
So $m(1\otimes S)(X')=0=\varepsilon(X)1$. Similarly, $m(S\otimes
1)(X')=0=\varepsilon(X)1$. The antipode axiom for  $Y$ is straightforward.
Part~(B) of Lemma~\ref{chegen} implies that $\widetilde
U(\mathfrak{a}\mathfrak{f}_1)_\hbar$ is a Hopf $\mathbin{\widehat{\otimes}}$-algebra.

Note also that the continuous antihomomorphism defined by
$$
 X\mapsto -X+\frac{\hbar }{\sinh \hbar}\,\sinh \hbar Y,\qquad Y\mapsto -Y,
$$
is inverse to~$S$.

Finally, it follows from \eqref{af1str} that  $\widetilde
U(\mathfrak{a}\mathfrak{f}_1)_\hbar$ is infinite-dimensional.
\end{proof}

\subsection*{A deformation of the function algebra of a non-algebraic complex Lie group}

Deformations of the  Hopf algebra of regular functions on some algebraic groups are well known. For example, in the classical case of semisimple groups \cite[Section~9.2.3]{KSc} the resulting  Hopf algebras are  affine and then their Arens--Michael envelopes are HFG Hopf algebras. The analytic form of a deformation of the function algebra on the complex Lie group $\mathrm{Af}_1$ of affine transformations of~$\mathbb{C}$ (the `az+b' group) is considered in \cite{Ak08}. We define here a deformation for a non-algebraic group, namely, the simply-connected covering $\widetilde{\mathrm{Af}}_1$ of $\mathrm{Af}_1$  and this example seems to be new.

Fix $q\in \mathbb{C}\setminus\{0\}$  and consider
the universal HFG algebra  generated by~$z$ and~$b$  subject to relation
$$
e^z b=q be^z.
$$
When $q=1$, it is isomorphic the algebra  of holomorphic functions
on  $\widetilde{\mathrm{Af}}_1$. For a general value of parameter, we
denote this HFG algebra by $\mathcal{O}(\widetilde{\mathrm{Af}}_1)_q$.

\begin{pr}\label{fuaf1q}
For every non-zero value of the parameter,
$\mathcal{O}(\widetilde{\mathrm{Af}}_1)_q$ is an infinite-dimensional  HFG  Hopf
algebra with respect to the comultiplication $\De$, counit $\varepsilon$
and antipode $S$ determined by
\begin{alignat*}{2}
\De\!&:z\mapsto 1\otimes z+ z\otimes 1,&\qquad
 b&\mapsto e^z\otimes b+ b\otimes 1\,;\\
S\!&:z\mapsto -z,&
b&\mapsto-e^{-z}b\,;\\
\varepsilon\!&:a,b\mapsto0\,.&&
\end{alignat*}
Moreover, $S$ is invertible.
\end{pr}

\begin{proof}
Put
$$
z'\!:=1\otimes z+ z\otimes 1, \qquad b'\!:=e^z\otimes b+ b\otimes 1.
$$
By Lemma~\ref{expprim}, we have  $e^{z'}= e^z\otimes e^z$. So
$$
e^{z'}b'=
(e^z\otimes e^z)(e^z\otimes b+ b\otimes 1)=q(e^z\otimes b+
b\otimes 1)(e^z\otimes e^z)=qb'e^{z'}.
$$
Hence $\De$ is a well-defined continuous homomorphism. It is easy
to see that $S$ is a well-defined continuous anti-automorphism
such that the antipode axiom holds on $\{z,b\}$. The similar fact
is true for~$\varepsilon$. Applying Lemma~\ref{chegen} we see that
$\mathcal{O}(\widetilde{\mathrm{Af}}_1)_q$ is an HFG Hopf algebra.

It is easy to see that $S^{-1}$ is defined by $z\mapsto -z$ and
$b\mapsto-e^{-z}b$.

To show that $\mathcal{O}(\widetilde{\mathrm{Af}}_1)_q$ is infinite dimensional
fix $\ze\in\mathbb{C}$ such that $e^\ze=q$ and consider the universal HFG algebra
$B$ generated by $x$ and $y$  subject to relation $[x,y]=\ze y$.
Then $(\mathop{\mathrm{ad}}\nolimits x)^n(y)=\ze^n y$ for every $n\in\mathbb{N}$. So,
by~\eqref{adhQ}, $[e^x,y]=(e^\ze-1)ye^x$.  Therefore $ e^x y=q
ye^x$ and, by the universal property, there is a continuous
homomorphism $\varphi\!:\mathcal{O}(\widetilde{\mathrm{Af}}_1)_q\to B$ that  maps~$z$ and~$b$ to~$e^x$ and~$y$, respectively. The structure of~$B$ is known:  there is a linear (but
not algebraic) topological isomorphism $B\cong  \mathbb{C}[[y]]\mathbin{\widehat{\otimes}}
\mathcal{O}(\mathbb{C})$ (\cite[Proposition~5.2]{Pir_qfree}; cf. \eqref{af1str}).
Obviously, the range of~$\varphi$ contains the subalgebra of~$B$ generated by~$y$, which is isomorphic to~$\mathbb{C}[y]$.
Therefore this subalgebra is infinite dimensional and so is
$\mathcal{O}(\widetilde{\mathrm{Af}}_1)_q$.
\end{proof}

\appendix
\section{Some properties of Stein algebras\\ and their dual spaces}
\label{ap:top}
Here we collected results which are used in proofs of Theorems~\ref{antiSTcoH} and~\ref{cicHFGgen}.

\subsection*{Auxiliary results for Theorem~\ref{antiSTcoH}}
Let $(X,\mathcal{O}_X)$  be a complex analytic space and $x\in X$. Consider the stalk
\begin{equation}\label{stalkd}
\mathcal{O}_{X,x}\!:=\varinjlim_{x\in U}\mathcal{O}_X(U),
\end{equation}
where $U$ runs over open subsets of $X$ containing $x$ and  $\mathcal{O}_X(U)$  denotes the algebra of holomorphic sections. Here the inductive limit (=colimit over a directed set) is considered in the category of rings. But we claim that it also can be interpreted in the category of commutative $\mathbin{\widehat{\otimes}}$-algebras.

Being a local analytic algebra (i.e., a quotient of the algebra of convergent power series), $\mathcal{O}_{X,x}$ can be endowed with the quotient topology (which is independent of the quotient homomorphism
\cite[Chapter 2, \S\,1, Corollary to Theorem~4]{GR1}). We need a result on this topology, which is
is well known in the case of a manifold (cf. \cite[\S\,2, Proposition, Part\,(iii)]{ADM71} and \cite[\S\,8]{Pir_stbflat}). But the author was unable to find a reference for a general complex analytic space. So a proof is included here for completeness. The main idea is the same as for manifolds; see  \cite[Chapter\,II, \S\,2, no.\,3, p.\,57, Proposition\,7]{Gr55} and \cite[\S\,4, (b) pp.\,80--81]{Gr54}. But we additionally need to pass to quotients.

Recall the definition of the standard topology on the algebra  $\mathcal{O}_X(U)$ of holomorphic sections.
Let $(U,\mathcal{O}_D/I)$ be a local model, i.e., $D$ is a polydisc in $\mathbb{C}^m$ for some  $m\in\mathbb{N}$, $\mathcal{O}_D$ is the sheaf of holomorphic functions in~$D$,  $I$ is a coherent sheaf ideal in $\mathcal{O}_D$, and $U$ is the support of~$I$. It follows from Cartan's theorem~B that $\mathcal{O}_X(U)\cong\mathcal{O}_D(D)/I(D)$. Since $I(D)$ is a closed ideal, $\mathcal{O}_X(U)$ is a Fr\'echet--Arens--Michael algebra. The definition can be extended to the case of an arbitrary $U$ but we do not need it. For some details see, e.g., the first three paragraphs of~\S\,1 in~\cite{Fo67}.

\begin{pr}\label{AMstalk}
Let $(X,\mathcal{O}_X)$  be a complex analytic space and $x\in X$.

\emph{(A)}~Suppose that all $\mathcal{O}_X(U)$ are endowed with their standard topologies. Then the topology of the locally convex inductive limit in~\eqref{stalkd} coincides with the quotient topology on $\mathcal{O}_{X,x}$.

\emph{(B)}~$\mathcal{O}_{X,x}$ is an Arens--Michael algebra.
\end{pr}

\begin{proof}
(A)~Consider the local model $(U,\mathcal{O}_D/I)$ as above. We can replace the inductive limit in~\eqref{stalkd} by  $\varinjlim\mathcal{O}_X(U_n)$, where $(U_n)$ is a decreasing sequence  of open neighbourhoods of $x$ such that $\bigcap_n U_n=\{x\}$. Moreover, we can put $U_n\!:=D_n\cap U$, where $(D_n)$ is a sequence of open polydiscs contained in $D$ such that the closure of $D_{n+1}$ belongs to $D_n$ and the polyradius tends to~$0$.  Then $\mathcal{O}_X(U_n)\cong\mathcal{O}_X(D_n)/I(D_n)$ as above.
Since in the category of locally convex spaces inductive limits commute with quotients, $\varinjlim\mathcal{O}_X(U_n)$ is isomorphic to $\mathcal{O}_{X,x}$ endowed with the quotient topology.

(B)~We now replace $(\mathcal{O}_X(U_n))$ by a sequence of Banach algebras.  Denote by $H^\infty(D_n)$ the Banach algebra of bounded holomorphic functions in $D_n$. Consider the natural homomorphism $\varphi_n\!:H^\infty(D_n)\to \mathcal{O}_{{\mathbb C}^m,x}$ and the quotient homomorphism $\alpha\!:\mathcal{O}_{{\mathbb C}^m,x}\to \mathcal{O}_{X,x}$. Then put $B_n\!:=H^\infty(D_n)/\Ker \alpha\varphi_n$. It is not hard to see that $(B_n)$ is an inductive sequence of Banach algebras with injective
compact linking maps. Moreover, its inductive limit in the category of locally convex spaces is isomorphic to $\mathcal{O}_{X,x}$; see \cite[Chapter~2, \S\,1, Theorem~7]{GR1} or \cite[Theorem 1.4]{Pa82}.

Being the inductive limit (as a locally convex space) of a sequence of normed algebras, $\mathcal{O}_{X,x}$  has jointly continuous multiplication and, moreover, is a locally $m$-convex algebra \cite{AN96} (for a short proof see~\cite[Theorem\,1]{DW97}).  Being the inductive limit of a sequence of Banach spaces with injective compact linking maps, $\mathcal{O}_{X,x}$ is complete \cite[Proposition\,7]{Si55}; see also a proof in \cite[\S\,26, 2.3, p.\,146, Korollar]{FW68}. Hence $\mathcal{O}_{X,x}$  is an Arens--Michael algebra.
\end{proof}

Note that even a strict inductive limit of Fr\'echet--Arens--Michael algebras can have discontinuous multiplication~\cite{DW97}. So the forgetful functor from  $\mathbin{\widehat{\otimes}}$-algebras  to locally convex spaces does not preserve in general even sequential inductive limits. Nevertheless, Proposition~\ref{AMstalk} asserts that the forgetful functor preserves the inductive limit under consideration and we use this fact in the proof of the following proposition.

\begin{pr}\label{indcattopa}
The forgetful functor from the category of commutative  $\mathbin{\widehat{\otimes}}$-algebras  to the category of commutative associative $\mathbb{C}$-algebras preserves the inductive limit of the system $(\mathcal{O}_X(U);\,U\ni x)$, where $x\in X$.
\end{pr}
\begin{proof}
It suffices to show that the inductive limit is preserved by the forgetful functor to the category of sets. Note that the forgetful functor from locally convex spaces to sets preserves inductive limits with injective linking maps (this follows from the explicit construction in \cite[Chapter\,II, \S\,6, 6.3]{SM}). Thus $\mathcal{O}_{X,x}$ can be considered as an inductive limit of locally convex spaces. But it follows from Proposition~\ref{AMstalk} that it is also a $\mathbin{\widehat{\otimes}}$-algebra. So $\mathcal{O}_{X,x}$ satisfies the desired universal property in the category of commutative $\mathbin{\widehat{\otimes}}$-algebras.
\end{proof}

If $X$ is a complex analytic space, we denote by $\mathcal{O}(X)$ the algebra of global sections of the sheaf $\mathcal{O}_X$. Take now two complex analytic  spaces $X$ and $Y$. The natural projections $X\times Y\to X$ and $X\times Y\to Y$ induce continuous homomorphisms  $\mathcal{O}(X)\to\mathcal{O}(X\times Y)$ and $\mathcal{O}(Y)\to\mathcal{O}(X\times Y)$. Since $\mathcal{O}(X)\mathbin{\widehat{\otimes}} \mathcal{O}(Y)$ is the coproduct of $\mathcal{O}(X)$ and $\mathcal{O}(Y)$ in the category of commutative $\mathbin{\widehat{\otimes}}$-algebras, there is a continuous homomorphisms
\begin{equation}\label{OOtimes}
\mathcal{O}(X)\mathbin{\widehat{\otimes}} \mathcal{O}(Y)\to
\mathcal{O}(X\times Y).
\end{equation}
\begin{pr}\label{OOtiso}
The homomorphism in~\eqref{OOtimes} is a topological isomorphism for every Stein spaces $X$ and $Y$.
\end{pr}
This result holds for arbitrary complex analytic spaces and goes back to Grothendieck; see \cite[II.3.3 and II.4.4]{Gr55}. The partial cases of polydiscs and simply connected domains can be found, e.g., in \cite[\S\,1.4]{Do74} and  \cite[Chapter~II, p.\,115, Corollary 4.15]{X1}, respectively. Here we include a short proof for Stein spaces based on universal properties and Forster's theorem.

\begin{proof}
Since $X\times Y$ is a product in the category of Stein spaces, Forster's theorem \cite{Fo67} implies that $\mathcal{O}(X\times Y)$ is a coproduct in the category of Stein algebras. So it suffices to show that $\mathcal{O}(X)\mathbin{\widehat{\otimes}} \mathcal{O}(Y)$ is a  Stein algebra.

Since $\mathcal{O}(X)$ and $\mathcal{O}(Y)$ are Stein algebras, they are quotients of $\mathcal{O}(\mathbb{C}^m)$ and $\mathcal{O}(\mathbb{C}^n)$ for some $m,n\in\mathbb{N}$, respectively. Since $\mathcal{O}(\mathbb{C}^m)$ and $\mathcal{O}(\mathbb{C}^n)$ are Fr\'echet spaces, $\mathcal{O}(\mathbb{C}^m)\mathbin{\widehat{\otimes}} \mathcal{O}(\mathbb{C}^n)\to \mathcal{O}(X)\mathbin{\widehat{\otimes}} \mathcal{O}(Y)$ is a  quotient map; see, e.g., \cite[\S\,41, p.\,189, (8)]{Kot2}. Finally note that $\mathcal{O}(\mathbb{C}^m)\mathbin{\widehat{\otimes}} \mathcal{O}(\mathbb{C}^n)\cong \mathcal{O}(\mathbb{C}^{m+m})$ and so $\mathcal{O}(X)\mathbin{\widehat{\otimes}} \mathcal{O}(Y)$ is a Stein algebra being quotient of $\mathcal{O}(\mathbb{C}^{m+m})$. This completes the proof.
\end{proof}

\subsection*{Auxiliary results for Theorem~\ref{cicHFGgen}}

\begin{lm}\label{Ptbam}
If $(X,\mathcal{O}_X)$  is a complex analytic space countable at infinity, then
$\mathcal{O}(X)$ is a nuclear Fr\'{e}chet space.
\end{lm}
\begin{proof}
If $\{U_n\!:\,n\in\mathbb{N}\}$ is an open cover of $X$ such that
$(U_n,\mathcal{O}_{U_n})$ is a local model of each $n$, then $\mathcal{O}(X)$ is
embedded in $\prod_n \mathcal{O}(U_n)$ as a closed subspace (cf.
\cite[\S~V.6.3]{GR3}).  By the definition of a local model, each
$\mathcal{O}(U_n)$ is a~quotient of the nuclear Fr\'{e}chet space
$\mathcal{O}(V)$, where $V$ is an open subset of $\mathcal{O}(\mathbb{C}^{m_n})$ for some  $m_n\in\mathbb{N}$. So it is also a nuclear Fr\'{e}chet space. Finally, the class of nuclear Fr\'{e}chet spaces is
stable under countable  products \cite[\S 18.3]{Kot1} and
passing to closed subspaces.
\end{proof}

We make use of two classes of locally convex spaces:
Pt\'{a}k (or $B$-complete)  and barreled spaces. The definitions can be found in \cite[\S~IV.8]{SM} and \cite[\S~II.7]{SM}, respectively; see also \cite[\S 34.2]{Kot2} and \cite[\S~27.1]{Kot1}.

\begin{lm}\label{Ptba1}
If $E$ is a nuclear Fr\'{e}chet space, then the strong dual space
$E'$ is a barreled Pt\'{a}k space.
\end{lm}

\begin{proof}
Any Fr\'{e}chet space is barreled,
any complete nuclear barreled space is a Montel space \cite[Corollary~3 of
Proposition~50.2]{Tre}, and any Montel space is reflexive \cite[Corollary  of
Proposition~36.9]{Tre}. Therefore $E$ is reflexive.

The strong dual of a reflexive
Fr\'{e}chet space is a Pt\'{a}k space   (\cite[\S 34.3(5)]{Kot2}
or~\cite[\S~IV.8, Example~2]{SM}); in particular, $E'$  is a Pt\'{a}k space.
The strong dual of a Montel space is a Montel space \cite[Proposition~36.10]{Tre},
and a~Montel space is barreled  by definition. Thus $E'$ is barreled.
\end{proof}

For arbitrary  complex analytic space $(X,\mathcal{O}_X)$  we consider the space ${\mathscr A}(X)\!:= \mathcal{O}(X)'$ of
analytic functionals on global sections.
Combining Lemmas~\ref{Ptbam} and~\ref{Ptba1} we get the following assertion.
\begin{pr}
\label{Ptba}
If $(X,\mathcal{O}_X)$  is a complex analytic space countable at infinity, then
${{\mathscr A}}(X)$ is a barreled Pt\'{a}k space.
\end{pr}

The following result is an application of this proposition.

\begin{pr}\label{opensurvl}
Let $N$ be a closed normal subgroup in a complex Lie group $G$ countable at infinity.
Then the $\mathbin{\widehat{\otimes}}$-algebra homomorphism $\varphi'\!:{\mathscr A}(G)\to{\mathscr A}(G/N)$
induced by the quotient map $G\to G/N$ is surjective and open.
\end{pr}
\begin{proof}
It is not hard to see that  the  continuous linear map
$\varphi\!:\mathcal{O}(G/N)\to \mathcal{O}(G)$ induced by $G\to G/N$ is
topologically injective. It follows from the
Hahn--Banach Theorem that the strong dual map $\varphi'$ is surjective \cite[II.4.2]{SM}. Proposition~\ref{Ptba} implies that both ${\mathscr A}(G)$ and ${\mathscr A}(G/N)$ are barreled Pt\'{a}k spaces.
So we can apply Pt\'{a}k's open mapping theorem, which asserts that every continuous surjective linear map from a
Pt\'{a}k space  to a barreled space  is open (see~\cite[\S 34.3]{Kot2} or~\cite[\S~IV.8.3, Corollary~1]{SM}). Thus $\varphi'$ is open.
\end{proof}

\section{Free products}
\label{ap:free}

Here some facts on free products of $\mathbin{\widehat{\otimes}}$-algebras and Arens--Michael algebras are collected.

\begin{pr}\label{fincopr}
\emph{(A)} The category of unital  $\mathbin{\widehat{\otimes}}$-algebras has finite coproducts.

\emph{(B)} The category of unital  Arens--Michael algebras has finite coproducts.
\end{pr}

Following tradition (for categories of rings or algebras) we say \emph{free
product} instead of `coproduct'. For simplicity, we restrict ourselves to the case of two algebras. The free
product of unital $\mathbin{\widehat{\otimes}}$-algebras $A$ and $B$ is denoted by $A\mathbin{\ast}B$; the free
product of unital Arens--Michael algebras $A$ and $B$ is denoted by $A\mathbin{\widehat\ast}B$.

The proof of Part~(B) can be found in  \cite[Proposition~4.2]{Pi15}.

\begin{proof}[Proof of Part~\emph{(A)}]
 Let  $A$ and $B$ be  unital $\mathbin{\widehat{\otimes}}$-algebras.
The explicit construction of $A\mathbin{\ast} B$ is as follows. For a given
complete locally convex space~$E$, consider the unital tensor $\mathbin{\widehat{\otimes}}$-algebra
$$
T(E)\!:=\mathbb{C}\oplus E\oplus\cdots  E^{\mathbin{\widehat{\otimes}} n}\oplus\cdots
$$
with obvious multiplication. (Recall that the direct sum in the category of complete locally convex spaces is the direct sum of linear spaces with the inductive topology.) Then $A\mathbin{\ast} B$ can be represented as
the completion of the quotient of $T(A\oplus B)$ by the closed two-sided ideal  generated  by all elements of the form $a\otimes a' - aa'$, $b\otimes b' - bb'$,  $1 - 1_A$, $1 - 1_B$, where    $a,a'\in A$, $b,b'\in B$ and $1_A$, $1_B$ are identities in $A$ and $B$, respectively.
\end{proof}

Remark that the explicit construction of $A\mathbin{\widehat\ast} B$ is similar, but the
topology on the corresponding tensor algebra is more complicated (see, for
example, \cite{Pi15}). A construction in the non-unital category is contained in a paper of Cuntz~\cite{Cu97}. Note that the categories in Proposition~\ref{fincopr} also  have arbitrary coproducts but we do not need this fact here (see \cite[Proposition~4.2]{Pi15} for the Arens--Michael case).

\begin{pr}\label{frpris}
\emph{(cf. \cite[Proposition 6.4]{Pir_stbflat} for tensor products)} For unital
$\mathbin{\widehat{\otimes}}$-algebras $A$ and $B$, the natural homomorphism
$$
(A\mathbin{\ast}B)\sphat\,\to\widehat A\mathbin{\widehat\ast} \widehat B\,
$$
is a topological isomorphism.
\end{pr}
\begin{proof}
The Arens--Michael enveloping functor is left adjoint to the forgetful
functor from the category of unital Arens--Michael algebras to the category
of unital $\mathbin{\widehat{\otimes}}$-algebras. So we can apply the general fact that a left
adjoint functor preserves colimits and, in particular, coproducts
\cite[Proposition~16.4.5]{Shu}.
\end{proof}

\begin{pr}\label{ga1APt}
Let $A$ and $B$ be unital $\mathbin{\widehat{\otimes}}$-algebras that are  nuclear (DF)-spaces. Then so is $A\mathbin{\ast}B$.
\end{pr}
For the proof we need the following lemma.
\begin{lm}\label{stnDF}
The class of complete nuclear (DF)-spaces is stable under finite complete projective tensor
products and countable direct sums.
\end{lm}
\begin{proof}
The assertion follows from standard properties of
(DF)-spaces and nuclear spaces; see \cite[p.~260, Theorem~12.4.8, p.~335, Theorem~15.6.2, and p.~483, Corollary~21.2.3]{Jar}.
\end{proof}

\begin{proof}[Proof of Proposition~\ref{ga1APt}]
By  construction, $A\mathbin{\ast}B$ is a completion of a quotient of the tensor algebra $T(A\oplus B)$ by a closed subspace. Since  $A$ and $B$   are complete  nuclear (DF)-spaces, so is $T(A\oplus B)$  by
Lemma~\ref{stnDF}. The class of (DF)-spaces is stable under completions and
quotients by a closed subspace \cite[p.~260, Theorem~12.4.8]{Jar}. Similarly, the stability property holds for the class of nuclear spaces [ibid. p.~483, Corollary~21.2.3]. Thus $A\mathbin{\ast}B$ is a nuclear (DF)-space.
\end{proof}

\end{document}